\documentclass[11pt,a4paper,reqno]{amsart}
\usepackage[british]{babel}

\usepackage{a4wide}
\usepackage{setspace}
\usepackage{amssymb}
\usepackage{amsmath}
\usepackage{amsthm}
\usepackage{enumitem}
\usepackage{mathrsfs}
\usepackage[colorlinks,citecolor=blue,urlcolor=black,linkcolor=red,linktocpage]{hyperref}
\usepackage{mathrsfs}
\usepackage{mathtools}

\newtheorem{thm}{Theorem}[section]

\newtheorem{cor}[thm]{Corollary}
\newtheorem{lem}[thm]{Lemma}

\newtheorem{prop}[thm]{Proposition}

\theoremstyle{definition}

\newtheorem{definition}[thm]{Definition}

\newtheorem{remark}[thm]{Remark}

\renewcommand{\epsilon}{\varepsilon}
\renewcommand{\phi}{\varphi}
\newcommand{\defeq}{\mathrel{\mathop:}=}

\DeclareMathOperator{\spt}{spt}
\DeclareMathOperator{\diam}{diam}

\begin{document}

\setlist{noitemsep}

\author{Friedrich Martin Schneider}
\address{F.M.S., Institute of Algebra, TU Dresden, 01062 Dresden, Germany }
\curraddr{Departamento de Matem\'atica, UFSC, Trindade, Florian\'opolis, SC, 88.040-900, Brazil}
\email{martin.schneider@tu-dresden.de}

\title[Equivariant concentration]{Equivariant concentration in topological groups}
\date{\today}


\begin{abstract} 
  We prove that, if $G$ is a second-countable topological group with a compatible right-invariant metric $d$ and $(\mu_{n})_{n \in \mathbb{N}}$ is a sequence of compactly supported Borel probability measures on $G$ converging to invariance with respect to the mass transportation distance over $d$ and such that $(\spt \mu_{n}, d\!\!\upharpoonright_{\spt \mu_{n}}, \mu_{n}\!\!\upharpoonright_{\spt \mu_{n}})_{n \in \mathbb{N}}$ concentrates to a fully supported, compact $mm$-space $(X,d_{X},\mu_{X})$, then $X$ is homeomorphic to a $G$-invariant subspace of the Samuel compactification of $G$. In particular, this confirms a conjecture by Pestov and generalizes a well-known result by Gromov and Milman on the extreme amenability of topological groups. Furthermore, we exhibit a connection between the average orbit diameter of a metrizable flow of an arbitrary amenable topological group and the limit of Gromov's observable diameters along any net of Borel probability measures UEB-converging to invariance over the group.
\end{abstract}


\maketitle


\section{Introduction}

Over the past few decades, the study of the measure concentration phenomenon has become a central theme in topological dynamics, in particular in the context of infinite-dimensional transformation groups. In fact, measure concentration ranges among the two most prominent pathways to extreme amenability, next to Ramsey-type phenomena~\cite{Pestov98,pestov02,KechrisPestovTodorcevic}. The origin of this development is marked by the groundbreaking work of Gromov and Milman~\cite{GromovMilman}, who showed that every \emph{L\'evy group}, i.e., a topological group containing a L\'evy family of compact subgroups with dense union, is extremely amenable, and who, moreover, exhibited a number of striking examples of such groups, e.g., the unitary group of the infinite-dimensional separable Hilbert space equipped with the strong operator topology. Their ideas were followed by numerous other examples, e.g.,~\cite{GiordanoPestov,Pestov07,CarderiThom}.

In his seminal work on metric measure geometry~\cite[Chapter~3$\tfrac{1}{2}$]{Gromov99}, Gromov offered a far-reaching extension of the measure concentration phenomenon: he introduced the \emph{observable distance} -- a metric on the set of isomorphism classes of \emph{$mm$-spaces}, i.e., separable complete metric spaces equipped with a Borel probability measure. The topology generated by this metric, the \emph{concentration topology}, captures the (classical) measure concentration phenomenon in a very natural manner: a sequence of $mm$-spaces constitutes a L\'evy family if and only if it \emph{concentrates} (converges in Gromov's observable distance) to a singleton space. But of course, the concentration topology allows for non-trivial limit objects, and recent years' growing interest in Polish groups with metrizable universal minimal flow~\cite{KechrisPestovTodorcevic,MNVTT,BMT} suggests to study manifestations of concentration to non-trivial spaces for topological groups. We attempt to advance this idea, originating in~\cite{PestovBook,GiordanoPestov}, with our main result.

\begin{thm}\label{theorem:equivariant.concentration} Let $G$ be a second-countable topological group equipped with a right-invariant compatible metric $d$. Suppose that there exists a sequence $(\mu_{n})_{n \in \mathbb{N}}$ of Borel probability measures on $G$ with compact supports $K_{n} \defeq \spt \mu_{n}$ $(n \in \mathbb{N})$ such that \begin{itemize}[leftmargin=11mm]
	\item[\textnormal{(A)}] $(\mu_{n})_{n \in \mathbb{N}}$ converges to invariance in the mass transportation distance over $d$, \smallskip
	\item[\textnormal{(B)}] $\left(K_{n}, d\!\!\upharpoonright_{K_{n}}, \mu_{n}\!\!\upharpoonright_{K_{n}}\right)_{n \in \mathbb{N}}$ concentrates to a fully supported, compact $mm$-space $(X,d_{X},\mu_{X})$.
\end{itemize} Then there exists a topological embedding $\psi \colon X \to \mathrm{S}(G)$ such that the push-forward measure $\psi_{\ast}(\mu_{X})$ is $G$-invariant. In particular, $\psi (X)$ is a $G$-invariant subspace of $\mathrm{S}(G)$. \end{thm}

\enlargethispage{2mm}

Due to recent work of the author and Thom~\cite[Theorem~3.2]{SchneiderThom}, every amenable second-countable topological group in fact admits a sequence of finitely supported probability measures converging to invariance with respect to the mass transportation distance over any right-invariant compatible metric. Furthermore, any Borel probability measure on a second-countable topological space assigns value $1$ to its support, and thus the restriction of the measure to the Borel $\sigma$-algebra of its support will indeed be a probability measure. In particular, this applies to the measures considered in condition~(B) of Theorem~\ref{theorem:equivariant.concentration}.

Since any minimal invariant closed subspace of the Samuel compactification $\mathrm{S}(G)$ of a topological group $G$ is a (up to isomorphism, the) universal minimal flow of $G$~\cite[Chapter~8]{auslander} (see also~\cite{ellis,deVries,usp}), the theorem above may be used to compute universal minimal flows of topological groups, or at least prove their metrizability. As universal minimal flows of non-compact locally compact groups are always non-metrizable~\cite[Theorem~A2.2]{KechrisPestovTodorcevic}, no such group can possibly satisfy the hypothesis of Theorem~\ref{theorem:equivariant.concentration}.

In particular, Theorem~\ref{theorem:equivariant.concentration} confirms a 2006 conjecture by Pestov~\cite[Conjecture~7.4.26]{PestovBook}: if $G$ is a metrizable topological group, equipped with a compatible right-invariant metric $d$, and $(K_{n})_{n \in \mathbb{N}}$ is an increasing sequence of compact subgroups such that \begin{itemize}
	\item[$\bullet$] the union $\bigcup_{n \in \mathbb{N}} K_{n}$ is everywhere dense in $G$, and
	\item[$\bullet$] $\left(K_{n}, d\!\!\upharpoonright_{K_{n}}, \mu_{n} \right)_{n \in \mathbb{N}}$ concentrates to a fully supported, compact $mm$-space $(X,d_{X},\mu_{X})$, where $\mu_{n}$ denotes the normalized Haar measure on $K_{n}$ $(n \in \mathbb{N})$,
\end{itemize} then the topological space $X$ supports the structure of a $G$-flow, with respect to which it admits a morphism to every $G$-flow. Indeed, it is easily seen that, by density of the increasing union of compact subgroups, the corresponding Haar measures converge to invariance in the mass transportation distance over $d$, whence Theorem~\ref{theorem:equivariant.concentration} asserts that $X$ is homeomorphic to a $G$-invariant subspace of $\mathrm{S}(G)$, which in turn gives rise to a $G$-flow on $X$ with the desired property, cf.~\cite[Corollary~3.1.12]{PestovBook}.

For a discussion of examples of the above kind of non-trivial concentration phenomenon, we refer to~\cite[Section~7]{GiordanoPestov} and~\cite[Chapter~7.4]{PestovBook}. In this connection, there is another intriguing question by Pestov~\cite[Problem~7.4.27]{PestovBook}: given a left-invariant metric~$d$ on the full symmetric group $\mathrm{Sym}(\mathbb{N})$ compatible with the topology of point-wise convergence, do the subgroups $(\mathrm{Sym}(n))_{n \in \mathbb{N}}$, equipped with their normalized counting measures and the restrictions of $d$, concentrate to the closed subspace $\mathrm{LO}(\mathbb{N}) \subseteq 2^{\mathbb{N} \times \mathbb{N}}$ of linear orders on $\mathbb{N}$, endowed with the unique $\mathrm{Sym}(\mathbb{N})$-invariant Borel probability measure~\cite{GlasnerWeiss02} and a suitable compatible metric? This question has been answered in the negative recently in~\cite{EquivariantDissipation}: in fact, the considered sequence of finite $mm$-spaces does not even admit a subsequence being Cauchy with respect to Gromov's observable distance.

In addition to Theorem~\ref{theorem:equivariant.concentration}, we will unveil another link between concentration phenomena and topological dynamics: roughly speaking, the average orbit diameter of an arbitrary flow of an amenable topological group, equipped with a continuous pseudo-metric, is bounded from above by the limit inferior of Gromov's observable diameters~\cite{Gromov99} (see Definition~\ref{definition:observable.diameters}) computed for any net of Borel probability measures on the acting group UEB-converging to invariance, with respect to the induced pseudo-metric. More precisely, if $G$ is a topological group and $X$ is a \emph{$G$-flow}, i.e., a non-empty compact Hausdorff space together with a continuous action of $G$ on it, then for any continuous pseudo-metric $d$ on $X$ and any point $x \in X$ there is a right-uniformly continuous pseudo-metric $d_{G,x}$ on $G$ defined by \begin{displaymath}
	d_{G,x}(g,h) \defeq d(gx,hx) \qquad (g,h \in G) ,
\end{displaymath} which is bounded from above by the bounded continuous right-invariant pseudo-metric \begin{displaymath}
	d_{G,X} \defeq \sup\nolimits_{y \in X} d_{G,y} \, \colon \, G \times G \, \longrightarrow \, \mathbb{R} .
\end{displaymath}

\begin{thm}\label{theorem:observable.diameters} Let $G$ be a topological group and let $(\mu_{i})_{i \in I}$ be a net of Borel probability measures on $G$ UEB-converging to invariance over $G$. If $d$ is a continuous pseudo-metric on a $G$-flow $X$ and $\nu$ is a $G$-invariant regular Borel probability measure on $X$, then \begin{displaymath}
	\int \sup\nolimits_{g \in E} d(x,gx) \, d\nu (x) \, \leq \, \sup\nolimits_{\alpha > 0} \liminf\nolimits_{i \to I} \sup\nolimits_{x \in X} \mathrm{ObsDiam}(G,d_{G,x},\mu_{i}; -\alpha) 
\end{displaymath} for every finite subset $E \subseteq G$. \end{thm}

\enlargethispage{2mm}

Let us add some remarks about Theorem~\ref{theorem:observable.diameters}. If $G$ is a topological group acting continuously on a compact Hausdorff space $X$ and $d$ is a continuous pseudo-metric on $X$, then \begin{displaymath}
	\sup\nolimits_{x \in X} \mathrm{ObsDiam}(G,d_{G,x},\mu; -\alpha) \, \leq \, \mathrm{ObsDiam}(G,d_{G,X},\mu; -\alpha)
\end{displaymath} for every Borel probability measure $\mu$ on $G$ and any $\alpha > 0$ (see~Remark~\ref{remark:monotonicity}), which means that Theorem~\ref{theorem:observable.diameters} immediately provides a corresponding estimate in terms of $d_{G,X}$. The same is true for Corollary~\ref{corollary:observable.diameters} below. Of course, if $G$ is separable, then Theorem~\ref{theorem:observable.diameters} asserts that \begin{displaymath}
	\int \sup\nolimits_{g \in G} d(x,gx) \, d\nu (x) \, \leq \, \sup\nolimits_{\alpha > 0} \liminf\nolimits_{i \to I} \sup\nolimits_{x \in X} \mathrm{ObsDiam}(G,d_{G,x},\mu_{i}; -\alpha) 
\end{displaymath} for any net $(\mu_{i})_{i \in I}$ of Borel probability measures on $G$ UEB-converging to invariance over~$G$. Moreover, Theorem~\ref{theorem:observable.diameters} readily implies~\cite[Theorem~3.9]{PestovSchneider}, our Corollary~\ref{corollary:whirly}, an extension of the result for Polish groups by Pestov~\cite[Theorem~5.7]{pestov10}, thus entailing earlier work of Glasner, Tsirelson and Weiss~\cite[Theorem~1.1]{GlasnerTsirelsonWeiss}, who showed that every spatial action of a L\'evy group must be trivial. We refer to the end of Section~\ref{section:observable.diameters} for a brief discussion on this. Furthermore, let us highlight another consequence of Theorem~\ref{theorem:observable.diameters}.

\begin{cor}\label{corollary:observable.diameters} Let $G$ be a topological group and let $(\mu_{i})_{i \in I}$ be a net of Borel probability measures on $G$ UEB-converging to invariance over $G$. If $d$ is a continuous pseudo-metric on a $G$-flow $X$, then there exists $x_{0} \in X$ such that \begin{displaymath}
	\sup\nolimits_{g \in G} d(x_{0},gx_{0}) \, \leq \, \sup\nolimits_{\alpha > 0} \liminf\nolimits_{i \to I} \sup\nolimits_{x \in X} \mathrm{ObsDiam}(G,d_{G,x},\mu_{i}; -\alpha) .
\end{displaymath} \end{cor}

The first estimates for orbit diameters, concerning H\"older actions on compact metric spaces, in terms of the isoperimetric behavior of the acting group and covering properties of the phase space belong to Milman~\cite{milman}. For generalizations of Milman's results as well as corresponding estimates for actions of L\'evy groups on a certain class of non-compact metric spaces, we refer to Funano's work~\cite{funano}.

Let us briefly outline the structure of the present article. In Section~\ref{section:gromov} we recollect some elementary facts and concepts concerning metrics and measures, leading up to the definition of Gromov's observable distance (Definition~\ref{definition:observable.distance}). In Section~\ref{section:invariance} we provide the background on UEB-convergence to invariance in topological groups necessary for the proof of Theorem~\ref{theorem:equivariant.concentration}, which is given in Section~\ref{section:equivariance}. Our final Section~\ref{section:observable.diameters} is devoted to proving Theorem~\ref{theorem:observable.diameters} and Corollary~\ref{corollary:observable.diameters}, as well as discussing some consequences. 

\section{Metrics, measures, and concentration}\label{section:gromov}

We start off by clarifying some notation. Given a set $X$, we denote by $\ell^{\infty}(X)$ the unital Banach algebra of all bounded real-valued functions on $X$ equipped with the supremum norm \begin{displaymath}
	\Vert f \Vert_{\infty} \defeq \sup \{ \vert f(x) \vert \mid x \in X \} \qquad (f \in \ell^{\infty}(X)) .
\end{displaymath}

Let $X$ be a topological space. If the topology of $X$ is generated by a metric $d$, then we call $d$ a \emph{compatible} metric on $X$. We will denote by $\mathrm{C}(X)$ the set of all continuous real-valued functions on $X$ and we let $\mathrm{CB}(X) \defeq \mathrm{C}(X) \cap \ell^{\infty}(X)$. Moreover, let us denote by $\mathcal{B}(X)$ the Borel $\sigma$-algebra of $X$ and by $\mathrm{P}(X)$ the set of all Borel probability measures on $X$. The \emph{weak topology} on $\mathrm{P}(X)$ is defined to be the initial topology on $\mathrm{P}(X)$ generated by the maps of the form $\mathrm{P}(X) \to \mathbb{R} , \, \mu \mapsto \int f \,d\mu$ where $f \in \mathrm{CB}(X)$. If $X$ is metrizable (or just perfectly normal), then the weak topology turns $\mathrm{P}(X)$ into a Tychonoff space. The \emph{support} of a measure $\mu \in \mathrm{P}(X)$ is defined as \begin{displaymath}
	\spt \mu \defeq \{ x \in X \mid \forall U \subseteq X \text{ open} \colon \, x \in U \Longrightarrow \mu (U) > 0\} ,
\end{displaymath} which is easily seen to form a closed subset of $X$. Given $\mu \in \mathrm{P}(X)$ and a Borel subset $B \subseteq X$ with $\mu (B) = 1$, we let $\mu \!\!\upharpoonright_{B} \, \defeq \mu\vert_{\mathcal{B}(B)} \in \mathrm{P}(B)$. The \emph{push-forward} of a measure $\mu \in \mathrm{P}(X)$ along a Borel map $f \colon X \to Y$ into another topological space $Y$ is defined to be \begin{displaymath}
	f_{\ast}(\mu) \colon \mathcal{B}(Y) \to [0,1], \quad B \mapsto \mu(f^{-1}(B)) .
\end{displaymath} Furthermore, let us note that each $\mu \in \mathrm{P}(X)$ gives rise to a pseudo-metric $\mathrm{me}_{\mu}$ on the set of all Borel measurable real-valued functions on $X$ defined by \begin{displaymath}
	\mathrm{me}_{\mu}(f,g) \defeq \inf \{ \epsilon > 0 \mid \mu (\{ x \in X \mid \vert f(x) - g(x) \vert > \epsilon \}) \leq \epsilon \} 
\end{displaymath} for any two Borel functions $f,g \colon X \to \mathbb{R}$.

Let $(X,d)$ be a pseudo-metric space. Given a subset $A \subseteq X$, we abbreviate $d\!\!\upharpoonright_{A} \, \defeq d\vert_{A \times A}$ and define $\diam (A,d) \defeq \sup \{ d(x,y) \mid x,y \in A \}$. For $x \in A \subseteq X$ and $\epsilon > 0$, we let \begin{align*}
	& B_{d}(x,\epsilon) \defeq \{ y \in X \mid d(x,y) < \epsilon \} , & B_{d}(A,\epsilon) \defeq \{ y \in X \mid \exists a \in A \colon \, d(a,y) < \epsilon \} .
\end{align*} Then the \emph{Hausdorff distance} between any two subsets $A,B \subseteq X$ is given by \begin{displaymath}
	d_{\mathrm{H}}(A,B) \defeq \inf \{ \epsilon > 0 \mid B \subseteq B_{d}(A,\epsilon), \, A \subseteq B_{d}(B,\epsilon) \} .
\end{displaymath} For $\ell, r \geq 0$, we denote by $\mathrm{Lip}_{\ell}(X,d)$ the set of all $\ell$-Lipschitz real-valued functions on $(X,d)$, and we define \begin{displaymath}
	\mathrm{Lip}_{\ell}^{\infty}(X,d) \defeq \mathrm{Lip}_{\ell}(X,d) \cap \ell^{\infty}(X) , \qquad \mathrm{Lip}_{\ell}^{r}(X,d) \defeq \{ f \in \mathrm{Lip}_{\ell}(X,d) \mid \Vert f \Vert_{\infty} \leq r \} .
\end{displaymath} Moreover, we let $\mathrm{Lip}(X,d) \defeq \bigcup \{ \mathrm{Lip}_{\ell}(X,d) \mid \ell \geq 0 \}$ and $\mathrm{Lip}^{\infty}(X,d) \defeq \mathrm{Lip}(X,d) \cap \ell^{\infty}(X)$. The \emph{mass transportation distance}\footnote{Different names appearing in the literature include \emph{Monge-Kontorovich distance}, \emph{bounded Lipschitz distance}, \emph{Wasserstein distance}, and \emph{Fortet-Mourier distance}, see~\cite{RR,GibbsSu,VillaniBook}.} $d_{\mathrm{MT}}$ over $d$ is the pseudo-metric on $\mathrm{P}(X)$ defined by \begin{displaymath}
	d_{\mathrm{MT}}(\mu,\nu) \defeq \sup\nolimits_{f \in \mathrm{Lip}_{1}^{1}(X,d)} \left\lvert \int f \, d\mu - \int f \, d\nu \right\rvert \qquad (\mu, \nu \in \mathrm{P}(X)) .
\end{displaymath} Furthermore, the \emph{Prokhorov distance} $d_{\mathrm{P}}$ over $d$ is the pseudo-metric on $\mathrm{P}(X)$ given by \begin{align*}
	d_{\mathrm{P}}(\mu,\nu) \ &\defeq \ \inf \{ \epsilon > 0 \mid \forall B \in \mathcal{B}(X) \colon \, \mu(B) \leq \nu (B_{d}(B,\epsilon)) + \epsilon \} \\
		& \ = \ \inf \{ \epsilon > 0 \mid \forall B \in \mathcal{B}(X) \colon \, \nu(B) \leq \mu (B_{d}(B,\epsilon)) + \epsilon \} \qquad (\mu, \nu \in \mathrm{P}(X)) .
\end{align*} In case that $(X,d)$ is a separable metric space, both $d_{\mathrm{P}}$ and $d_{\mathrm{MT}}$ are metrics compatible~with the weak topology on~$\mathrm{P}(X)$. For a more comprehensive account on these and other probability metrics, the reader is referred to~\cite{RR,GibbsSu,VillaniBook}.

Finally in this section, we will recall the very basics concerning Gromov's concentration topology~\cite[Chapter~3$\tfrac{1}{2}$.H]{Gromov99}, by following the presentation of Shioya~\cite[Chapter~5]{ShioyaBook}. This type of convergence refers to $mm$-spaces. An \emph{$mm$-space} is a triple $(X,d,\mu)$ where $(X,d)$ is a separable complete metric space and $\mu$ is a Borel probability measure on $X$. Moreover, an $mm$-space $(X,d,\mu)$ is called \emph{compact} if $(X,d)$ is compact, and \emph{fully supported} if $\spt \mu = X$. Henceforth, we will denote by $\lambda$ the Lebesgue measure on~$[0,1)$. A \emph{parametrization} of an $mm$-space $(X,d,\mu)$ is a Borel measurable map $\phi \colon [0,1) \to X$ such that $\phi_{\ast}(\lambda) = \mu$. It is well known that any $mm$-space admits a parametrization (see, e.g.,~\cite[Lemma~4.2]{ShioyaBook}).

\begin{definition}\label{definition:observable.distance} The \emph{observable distance} between two $mm$-spaces $X$ and $Y$ is defined to be \begin{displaymath}
	d_{\mathrm{conc}}(X,Y) \defeq \inf \{ (\mathrm{me}_{\lambda})_{\mathrm{H}}(\mathrm{Lip}_{1}(X) \circ \phi, \mathrm{Lip}_{1}(Y) \circ \psi) \mid \phi \text{ param.~of } X, \, \psi \text{ param.~of } Y \} .
\end{displaymath} A sequence of $mm$-spaces $(X_{n})_{n \in \mathbb{N}}$ is said to \emph{concentrate to} an $mm$-space $X$ if \begin{displaymath}
\lim\nolimits_{n \to \infty} d_{\mathrm{conc}}(X_{n},X) = 0 .
\end{displaymath} \end{definition}

It is known that the observable distance induces a metric on the set of isomorphism classes of $mm$-spaces, cf.~\cite[Theorem~5.16]{ShioyaBook}. In particular, two $mm$-spaces $X$ and $Y$ are \emph{isomorphic}, i.e., there exists an $mm$-space isomorphism between  $X$ and $Y$, if and only if $d_{\mathrm{conc}}(X,Y) = 0$. By an \emph{isomorphism} between $mm$-spaces $(X,d_{X},\mu_{X})$ and $(Y,d_{Y},\mu_{Y})$ we mean an isometry \begin{displaymath}
	f \colon \left(\spt \mu_{X},d_{X}\! \! \upharpoonright_{\spt \mu_{X}}\right) \to \left(\spt \mu_{Y},d_{Y}\! \! \upharpoonright_{\spt \mu_{Y}}\right)
\end{displaymath} such that $f_{\ast}\! \left(\mu_{X}\! \! \upharpoonright_{\spt \mu_{X}}\right) = \mu_{Y}\! \! \upharpoonright_{\spt \mu_{Y}}$. For our purposes, i.e., the proof of our Theorem~\ref{theorem:equivariant.concentration}, the following characterization of concentration will be useful.

\begin{thm}[\cite{ShioyaBook}, Corollary~5.35]\label{theorem:measure.concentration} A sequence of $mm$-spaces $(X_{n},d_{n},\mu_{n})_{n \in \mathbb{N}}$ concentrates to an $mm$-space $(X,d,\mu)$ if and only if there is a sequence of Borel maps $p_{n} \colon X_{n} \to X$ $(n \in \mathbb{N})$ such that \begin{enumerate}
	\item[$(1)$] $(p_{n})_{\ast}(\mu_{n}) \longrightarrow \mu$ in the weak topology as $n \to \infty$,
	\item[$(2)$] $(\mathrm{me}_{\mu_{n}})_{\mathrm{H}}(\mathrm{Lip}_{1}(X,d) \circ p_{n},\mathrm{Lip}_{1}(X_{n},d_{n})) \longrightarrow 0$ as $n \to \infty$.
\end{enumerate} \end{thm}

\section{Topological groups and convergence to invariance}\label{section:invariance}

In this section we briefly recollect some results from~\cite{SchneiderThom} about UEB-convergence to invariance over topological groups. Throughout the present note, by a \emph{topological group} we will always mean a \emph{Hausdorff} topological group.

Let $X$ be a uniform space. Consider the commutative unital real Banach algebra $\mathrm{UCB}(X)$ of all bounded uniformly continuous real-valued functions on $X$ endowed with the supremum norm. The set $\mathrm{M}(X)$ of all \emph{means} on $\mathrm{UCB}(X)$, i.e., (necessarily continuous) positive unital linear maps from $\mathrm{UCB}(X)$ to $\mathbb{R}$, equipped with the weak-$\ast$ topology, i.e., the initial topology generated by the maps of the form $\mathrm{M}(X) \to \mathbb{R}, \, \mu \mapsto \mu (f)$ where $f \in \mathrm{UCB}(X)$, constitutes a compact Hausdorff space. The set $\mathrm{S}(X)$ of all (necessarily positive and linear) unital ring homomorphisms from $\mathrm{UCB}(X)$ into~$\mathbb{R}$ forms a closed subspace of $\mathrm{M}(X)$, which is called the \emph{Samuel compactification} of~$X$. The map $\eta_{X} \colon X \to \mathrm{S}(X)$ given by \begin{displaymath}
	\eta_{X}(x)(f) \defeq f(x) \qquad (x \in X, \, f \in \mathrm{UCB}(X)) 
\end{displaymath} is uniformly continuous and has dense range in $\mathrm{S}(X)$, and the mapping \begin{displaymath}
	\mathrm{C}(\mathrm{S}(X)) \, \longrightarrow \, \mathrm{UCB}(X), \qquad f \, \longmapsto \, f \circ \eta_{X}
\end{displaymath} is an isometric isomorphism of unital Banach algebras. Furthermore, a subset $H \subseteq \mathrm{UCB}(X)$ is called \emph{UEB} (short for \emph{uniformly equicontinuous bounded}) if $H$ is norm-bounded and \emph{uniformly equicontinuous}, i.e., for every $\epsilon > 0$ there exists an entourage $U$ of $X$ such that \begin{displaymath}
	\forall f \in H \ \forall (x,y) \in U \colon \qquad \vert f(x) - f(y) \vert \, \leq \, \epsilon .
\end{displaymath} The collection $\mathrm{UEB}(X)$ of all UEB subsets of $\mathrm{UCB}(X)$ constitutes a convex vector bornology on the vector space $\mathrm{UCB}(X)$. It is easily seen that a subset $H \subseteq \mathrm{UCB}(X)$ belongs to $\mathrm{UEB}(X)$ if and only if $H$ is norm-bounded and there is a uniformly continuous pseudo-metric $d$ on $X$ such that $H \subseteq \mathrm{Lip}_{1}(X,d)$. The \emph{UEB topology} on the continuous dual~$\mathrm{UCB}(X)^{\ast}$ is defined as the topology of uniform convergence on UEB subsets of $\mathrm{UCB}(X)$. This is a locally convex linear topology on the vector space $\mathrm{UCB}(X)^{\ast}$ containing the weak-$\ast$ topology, i.e., the initial topology generated by the maps $\mathrm{UCB}(X)^{\ast} \to \mathbb{R}, \, \mu  \mapsto \mu (f)$ where $f \in \mathrm{UCB}(X)$. For more details on the UEB topology, the reader is referred to~\cite{PachlBook}.

Now let $G$ be a topological group. Denote by $\mathcal{U}(G)$ the neighborhood filter of the neutral element in $G$ and endow $G$ with its \emph{right uniformity} defined by the basic entourages \begin{displaymath}
	\left\{ (x,y) \in G \times G \left\vert \, yx^{-1} \in U \right\} \qquad (U \in \mathcal{U}(G)) . \right.
\end{displaymath} Referring to the right uniformity, we denote by $\mathrm{RUCB}(G)$ the set of all bounded uniformly continuous real-valued function on $G$ and by $\mathrm{RUEB}(G)$ the set of all UEB subsets of $\mathrm{RUCB}(G)$. It is easily seen that a subset $H \subseteq \mathrm{RUCB}(G)$ belongs to $\mathrm{RUEB}(G)$ if and only if $H$ is norm-bounded and there is a continuous right-invariant pseudo-metric $d$ on $G$ with $H \subseteq \mathrm{Lip}_{1}(G,d)$. Furthermore, for $g \in G$, we let $\lambda_{g} \colon G \to G, \, x \mapsto gx$ and $\rho_{g} \colon G \to G, \, x \mapsto xg$. We note that $G$ acts continuously on $\mathrm{M}(G)$ by \begin{displaymath}
	(g\mu)(f) \defeq \mu (f \circ \lambda_{g}) \qquad (g \in G, \, \mu \in \mathrm{M}(G), \, f \in \mathrm{RUCB}(G)) ,
\end{displaymath} and that $\mathrm{S}(G)$ constitutes a $G$-invariant subspace of $\mathrm{M}(G)$. Let us recall that $G$ is \emph{amenable} (resp., \emph{extremely amenable}) if $\mathrm{M}(G)$ (resp., $\mathrm{S}(G)$) admits a $G$-fixed point. It is well known that $G$ is amenable (resp., extremely amenable) if and only if every $G$-flow admits a $G$-invariant regular Borel probability measure (resp., a $G$-fixed point). For a more comprehensive account on (extreme) amenability of general topological groups, we refer to~\cite{PestovBook}.

We will need a characterization of amenability in terms of almost invariant finitely supported probability measures from~\cite{SchneiderThom}.

\begin{definition} Let $G$ be a topological group. A net $(\mu_{i})_{i \in I}$ of Borel probability measures on $G$ is said to \emph{UEB-converge to invariance (over $G$)} if \begin{displaymath}
	\forall g \in G \ \forall H \in \mathrm{RUEB}(G) \colon \qquad \sup\nolimits_{f \in H} \left\lvert \int f \circ \lambda_{g} \, d\mu_{i} - \int f \, d\mu_{i} \right\rvert \longrightarrow 0 \quad (i \to I) .
\end{displaymath} \end{definition}

\begin{thm}[\cite{SchneiderThom}, Theorem~3.2]\label{theorem:topological.day} A topological group is amenable if and only if it admits a net of (finitely supported regular) Borel probability measures UEB-converging to invariance. \end{thm}

We note some elementary properties of UEB-convergence to invariance.

\begin{lem}\label{lemma:translated.convergence} Let $G$ be a topological group. Let $(\mu_{i})_{i \in I}$ be a net of Borel probability measures UEB-converging to invariance over $G$. The following hold. \begin{itemize}
	\item[$(1)$] For any $(g_{i})_{i \in I} \in G^{I}$, the net $((\rho_{g_{i}})_{\ast}(\mu_{i}))_{i \in I}$ UEB-converges to invariance over $G$.
	\item[$(2)$] If $\phi \colon G \to H$ is a continuous homomorphism with dense range in a topological group~$H$, then $(\phi_{\ast}(\mu_{i}))_{i \in I}$ UEB-converges to invariance over $H$.
	\item[$(3)$] For each $i \in I$, let $C_{i}$ be a Borel subset of $G$ such that $\mu_{i}(C_{i}) = 1$. Then $\bigcup_{i \in I} C_{i}C_{i}^{-1}$ is dense in $G$.
\end{itemize} \end{lem}

\begin{proof} (1) Let $(g_{i})_{i \in I} \in G^{I}$. Consider any $H \in \mathrm{RUEB}(G)$. Then it is straightforward to check that $\{ f \circ \rho_{g_{j}} \mid f \in F, \, j \in I \} \in \mathrm{RUEB}(G)$. Hence, for every $g \in G$, \begin{displaymath}
	\sup\nolimits_{(f,j) \in F \times I} \left\lvert \int f \circ \rho_{g_{j}} \, d\mu_{i} - \int f \circ \rho_{g_{j}} \circ \lambda_{g} \, d\mu_{i} \right\rvert \, \longrightarrow \, 0 \quad (i \to I) 
\end{displaymath} and, in particular, \begin{align*}
	\sup\nolimits_{f \in F} \left\lvert \int f \, d(\rho_{g_{i}})_{\ast}(\mu_{i}) \right. &\left. - \int f \circ \lambda_{g} \, d(\rho_{g_{i}})_{\ast}(\mu_{i}) \right\rvert \, = \, \sup\nolimits_{f \in F} \left\lvert \int f \circ \rho_{g_{i}} \, d\mu_{i} -\int f \circ \lambda_{g} \circ \rho_{g_{i}} \, d\mu_{i} \right\rvert \\
	&= \, \sup\nolimits_{f \in F} \left\lvert \int f \circ \rho_{g_{i}} \, d\mu_{i} -\int f \circ \rho_{g_{i}} \circ \lambda_{g} \, d\mu_{i} \right\rvert \, \longrightarrow \, 0 \quad (i \to I) .
\end{align*}
	
(2) Let $h \in H$ and $F \in \mathrm{RUEB}(H)$. We wish to show that \begin{displaymath}
	\sup\nolimits_{f \in F} \left\lvert \int f \, d\phi_{\ast}(\mu_{i}) - \int f \circ \lambda_{h} \, d \phi_{\ast}(\mu_{i}) \right\rvert \, \longrightarrow \, 0 \quad (i \to I) .
\end{displaymath} Let $\epsilon > 0$. Since $F$ belongs to $\mathrm{RUEB}(H)$, there exists $U \in \mathcal{U}(H)$ such that $\Vert f - (f \circ \lambda_{u}) \Vert_{\infty} \leq \tfrac{\epsilon}{2}$ for all $f \in F$ and $u \in U$. Due to $\phi (G)$ being dense in $H$, there exists $g \in G$ with $\phi (g) \in Uh$. As $\phi \colon G \to H$ is uniformly continuous with regard to the respective right uniformities, it follows that $F\circ \phi \in \mathrm{RUEB}(G)$. Hence, we find $i _{0} \in I$ such that \begin{displaymath}
	\forall i \in I , \, i \geq i_{0} \colon \quad \sup\nolimits_{f \in F} \left\lvert \int f \circ \phi \, d\mu_{i} - \int f \circ \phi \circ \lambda_{g} \, d\mu_{i} \right\rvert \, \leq \, \tfrac{\epsilon}{2} . 
\end{displaymath} For every $i \in I$ with $i \geq i_{0}$, we conclude that \begin{align*}
	\left\lvert \int f \,d\phi_{\ast}(\mu_{i}) \right. &\left. - \int f \circ \lambda_{h} \, d\phi_{\ast}(\mu_{i}) \right\rvert \, = \, \left\lvert \int f \circ \phi \, d\mu_{i} -\int f \circ \lambda_{h} \, d\phi_{\ast}(\mu_{i}) \right\rvert \\
	&\leq \, \left\lvert \int f \circ \phi \, d\mu_{i} -\int f \circ \phi \circ \lambda_{g} \, d\mu_{i} \right\rvert + \left\lvert \int f \circ \lambda_{\phi (g)} \, d\phi_{\ast}(\mu_{i}) -\int f \circ \lambda_{h} \, d\phi_{\ast}(\mu_{i}) \right\rvert \\
	&\leq \, \tfrac{\epsilon}{2} + \left\lVert \left(f \circ \lambda_{\phi (g)}\right) - \left(f \circ \lambda_{h}\right) \right\rVert_{\infty} \, = \, \tfrac{\epsilon}{2} + \left\lVert f - \left(f \circ \lambda_{\phi (g)h^{-1}}\right) \right\rVert_{\infty} \, \leq \, \epsilon 
\end{align*} for all $f \in F$, i.e., $\sup_{f \in H} \left\lvert \int f \, d\phi_{\ast}(\mu_{i}) -\int f \circ \lambda_{h} \, d\phi_{\ast}(\mu_{i}) \right\rvert \leq \epsilon$ as desired.

(3) Let $C \defeq \bigcup_{i \in I} C_{i}C_{i}^{-1}$. Consider any $g \in G$ and $U \in \mathcal{U}(G)$. We are going to show that $gU \cap C \ne \emptyset$. By Urysohn's lemma for uniform spaces, there exists a right-uniformly continuous function $f \colon G \to [0,1]$ such that $f(e) = 1$ and $f(x) = 0$ whenever $x \in G\setminus U$. For every subset $S \subseteq G$, define $f_{S} \colon G \to [0,1]$ by \begin{displaymath}
	f_{S}(x) \defeq \sup\nolimits_{s \in S} f\! \left(xs^{-1}\right) \qquad (x \in G) .
\end{displaymath} It is straightforward to check that the set $\{ f_{S} \mid S \subseteq G \}$ belongs to $\mathrm{RUEB}(G)$. Since $(\mu_{i})_{i \in I}$ UEB-converges to invariance over $G$, there exists $i_{0} \in I$ such that \begin{displaymath}
	\forall i \in I, \, i \geq i_{0} \colon \quad \sup\nolimits_{S \subseteq G} \left\lvert \int f_{S} \, d\mu_{i} - \int f_{S} \circ \lambda_{g^{-1}} \, d\mu_{i} \right\rvert \, \leq \, \tfrac{1}{2} .
\end{displaymath} Let $i \in I$. We observe that $\int f_{C_{i}} \, d\mu_{i} = 1$, as $\mu_{i}(C_{i}) = 1$ and $C_{i} \subseteq f_{C_{i}}^{-1}(1)$. Hence, if $i \geq i_{0}$, then $\int f_{C_{i}} \circ \lambda_{g^{-1}} \, d\mu_{i} \geq \tfrac{1}{2}$ and so $\left(f_{C_{i}} \circ \lambda_{g^{-1}}\right)\vert_{C_{i}} \neq \mathbf{0}$, thus $gUC_{i} \cap C_{i} \ne \emptyset$. This entails that $gU \cap C \ne \emptyset$. Consequently, $C$ is dense in $G$. \end{proof}

Let us note the following consequence of Lemma~\ref{lemma:translated.convergence}(3).

\begin{cor}\label{corollary:second.countable} If a metrizable topological group $G$ admits a sequence $(\mu_{n})_{n \in \mathbb{N}}$ of Borel probability measures UEB-converging to invariance such that, for each $n \in \mathbb{N}$, there is a compact subset $C_{n} \subseteq G$ with $\mu_{n}(C_{n}) = 1$, then $G$ is separable. \end{cor}

For metrizable topological groups, one may reformulate UEB-convergence to invariance in terms of mass transportation distances over compatible right-invariant metrics, see Corollary~\ref{corollary:wasserstein.convergence.to.inavriance}. This will be a consequence of the following fact about metrizable uniformities.

\begin{lem}\label{lemma:lipschitz.approximation} Let $(X,d)$ be a metric space and $H \in \mathrm{UEB}(X,d)$. For every $\epsilon > 0$, \begin{displaymath}
	\exists \ell \geq 1 \ \forall f \in H \ \exists f' \in \mathrm{Lip}_{\ell}^{\ell}(X,d) \colon \qquad \left\lVert f - f' \right\rVert_{\infty} \leq  \epsilon .
\end{displaymath} \end{lem}

\begin{proof} Upon translating $H$ by a suitable constant function, we may and will assume that $f \geq 0$ for all $f \in H$. Put $s \defeq \sup_{f \in H} \Vert f \Vert_{\infty}$. Let $\epsilon \in (0,1]$. Since $H$ is uniformly equicontinuous, we find $\delta > 0$ such that \begin{displaymath}
	\forall f \in H \ \forall x,y \in X \colon \qquad d(x,y) < \delta\, \Longrightarrow \, \vert f(x) -f(y) \vert \leq \epsilon .
\end{displaymath} Let $k \defeq \tfrac{s + \epsilon}{\delta}$ and $\ell \defeq \max \{ k, s+1 \}$. For each $f \in H$, define $f_{k} \colon X \to \mathbb{R}$ by \begin{displaymath}
	f_{k}(x) \defeq \inf\nolimits_{y \in X} f(y) + kd(x,y) \qquad (x \in X) .
\end{displaymath} Note that $f_{k} \colon (X,d) \to \mathbb{R}$ is $k$-Lipschitz for every $f \in H$. Let us prove that $\Vert f - f_{k} \Vert_{\infty} \leq \epsilon$ for all $f \in H$. For this purpose, let $f \in H$. Since $f_{k} \leq f$, it suffices to show that $f_{k} \geq f - \epsilon$.  To this end, let $x \in X$. For each $y \in X$, either $f(y) \geq f(x) - \epsilon$ and therefore \begin{displaymath}
	f(y) + kd(x,y) \, \geq \, f(y) \, \geq \, f(x) -\epsilon ,
\end{displaymath} or $f(y) < f(x) - \epsilon$ and thus $d(x,y) \geq \delta$, which entails that \begin{displaymath}
	f(y) + kd(x,y) \, \geq \, k\delta \, = \, s + \epsilon \, \geq \, f(x) + \epsilon .
\end{displaymath} In any case, $f(y) + kd(x,y) \geq f(x) - \epsilon$ for all $y \in X$. Consequently, $f_{k}(x) \geq f(x) - \epsilon$ as desired. In turn, $\Vert f_{k} \Vert_{\infty} \leq \Vert f \Vert_{\infty} + \epsilon \leq s + 1$ and hence $f_{k} \in \mathrm{Lip}_{\ell}^{\ell}(X,d)$. \end{proof}

\begin{cor}\label{corollary:wasserstein.convergence.to.inavriance} Let $G$ be a topological group and let $d$ be a compatible right-invariant metric on $G$. A net $(\mu_{i})_{i \in I}$ of Borel probability measures UEB-converges to invariance over $G$ if and only if $(\mu_{i})_{i \in I}$ converges to invariance in the mass transportation distance over $d$, i.e., \begin{displaymath}
	\forall g \in G \colon \qquad d_{\mathrm{MT}}((\lambda_{g})_{\ast}(\mu_{i}),\mu_{i}) \longrightarrow 0 \quad (i \to I) .
\end{displaymath} \end{cor}

\begin{proof} Since $d$ is continuous and right-invariant, $\mathrm{Lip}_{1}^{1}(G,d)$ belongs to $\mathrm{RUEB}(G)$, whence the former implies the latter. To prove the converse, let $(\mu_{i})_{i \in I}$ be a net of Borel probability measures converging to invariance in the mass transportation distance over $d$. Given that $d$ is right-invariant and generates the topology of $G$, it is easily seen that the right uniformity of~$G$ coincides with the uniformity induced by $d$. Let $H \in \mathrm{RUEB}(G) = \mathrm{UEB}(X,d)$ and $g \in G$. Consider any $\epsilon > 0$. Thanks to~Lemma~\ref{lemma:lipschitz.approximation}, there exists $\ell \geq 1$ with \begin{displaymath}
	H \, \subseteq \, B_{\Vert \cdot \Vert_{\infty}}\! \left(\mathrm{Lip}_{\ell}^{\ell}(G,d),\tfrac{\epsilon}{3}\right) .
\end{displaymath} By assumption, we find $i_{0} \in I$ so that \begin{displaymath}
	\forall i \in I, \, i\geq i_{0} \colon \qquad \sup\nolimits_{f \in \mathrm{Lip}_{1}^{1}(G,d)} \left\lvert \int f \circ \lambda_{g} \,d\mu_{i} - \int f \, d\mu_{i} \right\rvert \, \leq \, \tfrac{\epsilon}{3\ell}
\end{displaymath} Let $i \in I$ with $i \geq i_{0}$. For each $f \in H$, there exists $f' \in \mathrm{Lip}_{\ell}^{\ell}(G,d) = \ell \cdot \mathrm{Lip}_{1}^{1}(G,d)$ with $\lVert f-f' \rVert_{\infty} \leq \tfrac{\epsilon}{3}$, and thus \begin{align*}
	\left\lvert \int f \circ \lambda_{g} \, d\mu_{i} - \int f \, d\mu_{i} \right\rvert \, &\leq \, \lVert (f \circ \lambda_{g}) - (f' \circ \lambda_{g}) \rVert_{\infty} + \left\lvert \int f' \circ \lambda_{g} \, d\mu_{i} - \int f' \, d\mu_{i} \right\rvert + \lVert f' - f \rVert_{\infty} \\
	&\leq \, \tfrac{\epsilon}{3} + \ell \tfrac{\epsilon}{3\ell} + \tfrac{\epsilon}{3} \, = \, \epsilon ,
\end{align*} i.e., $\sup_{f \in H} \left\lvert \int f \circ \lambda_{g} \, d\mu_{i} - \int f \, d\mu_{i} \right\rvert \leq \epsilon$. So, $(\mu_{i})_{i \in I}$ UEB-converges to invariance over $G$. \end{proof}

\section{Equivariant concentration}\label{section:equivariance}

Our proof of Theorem~\ref{theorem:equivariant.concentration} will make a distinction between the precompact and the non-precompact case. Whereas the former may be settled by a very simple and straightforward argument, the treatment of the latter is somewhat more complicated. Recall that a topological group $G$ is said to be \emph{precompact} if, for every $U \in \mathcal{U}(G)$, there exists a finite subset $F \subseteq G$ such that $G = UF$. It is well known that a topological group is precompact if and only if it embeds into a compact group (see~\cite[Corollary~3.7.17]{AT}). The following characterization of precompact groups was obtained independently by Uspenskij (unpublished, cf.~a footnote in~\cite{uspenskij}) and Solecki~\cite{solecki}. A short proof may be found in~\cite[Proposition~4.3]{BouziadTroallic}.

\begin{lem}\label{lemma:usp} Let $G$ be a topological group. If for every $U \in \mathcal{U}(G)$ there exists a finite subset $F \subseteq G$ with $G = FUF$, then $G$ is precompact. \end{lem}

We will need the above in the form of Corollary~\ref{corollary:usp.2}.

\begin{cor}\label{corollary:usp} Let $G$ be a topological group. If for every $U \in \mathcal{U}(G)$ there exists a compact subset $K \subseteq G$ with $G = KUK$, then $G$ is precompact. \end{cor}

\begin{proof} We apply Lemma~\ref{lemma:usp}. Let $U \in \mathcal{U}(G)$. Pick $V \in \mathcal{U}(G)$ with $V^{3} \subseteq U$. By assumption, we find a compact subset $K \subseteq G$ such that $G = KVK$. Since $K$ is compact, there is a finite set $F \subseteq G$ with $K \subseteq VF$ and $K \subseteq FV$. It follows that $G = KVK \subseteq FV^{3}F \subseteq FUF$. \end{proof}

\begin{cor}\label{corollary:usp.2} Let $G$ be a topological group. If $G$ is not precompact, then there is $U \in \mathcal{U}(G)$ such that, for every sequence $(K_{n})_{n \in \mathbb{N}}$ of compact subsets of $G$, there is $(g_{n})_{n \in \mathbb{N}} \in G^{\mathbb{N}}$ with \begin{displaymath}
	\forall m,n \in \mathbb{N}, \, m \ne n \colon \qquad UK_{m}g_{m} \cap UK_{n}g_{n} = \emptyset .
\end{displaymath} \end{cor}
	
\begin{proof} Since $G$ is not precompact, Corollary~\ref{corollary:usp} above asserts the existence of some $V \in \mathcal{U}(G)$ such that $G \ne KVK$ for any compact set $K \subseteq G$. Let $U \in \mathcal{U}(G)$ with $U^{-1}U \subseteq V$. We claim that $U$ has the desired property. To see this, let $(K_{n})_{n \in \mathbb{N}}$ be a sequence of compact subsets of $G$. We select $(g_{n})_{n \in \mathbb{N}} \in G^{\mathbb{N}}$ recursively as follows: we let $g_{0} \defeq e$, and if $g_{0},\ldots,g_{n-1} \in G$ are chosen appropriately, then we pick \begin{displaymath}
	g_{n} \in G \setminus (K^{-1}_{n}V(K_{0}g_{0} \cup \ldots \cup K_{n-1}g_{n-1}))
\end{displaymath} and note that $K_{n}g_{n} \cap V(K_{0}g_{0} \cup \ldots \cup K_{n-1}g_{n-1}) = \emptyset$, whence \begin{displaymath}
	UK_{n}g_{n} \cap U(K_{0}g_{0} \cup \ldots \cup K_{n-1}g_{n-1}) = \emptyset .
\end{displaymath} Evidently, the sequence $(g_{n})_{n \in \mathbb{N}}$ is as desired. \end{proof}

Before moving on to the proof of Theorem~\ref{theorem:equivariant.concentration}, let us note two basic preliminary observations (Lemma~\ref{lemma:extending.metrics} and Lemma~\ref{lemma:density}). The first one, concerning the metrizability of topological groups, will be deduced from the following more general fact. 

\begin{lem}\label{lemma:extending.uniform.metrics} Let $X$ be a dense subset of a Hausdorff uniform space $Y$. Suppose that $X$ admits a metric $d$ generating the subspace uniformity inherited from $Y$. Then there exists a unique continuous map $D \colon Y \times Y \to \mathbb{R}$ with $D\vert_{X \times X} = d$, and furthermore $D$ is a metric generating the uniformity of $Y$. \end{lem}

\begin{proof} Uniqueness of the desired map is an immediate consequence of $X$ being dense in $Y$. Let us prove the existence. Since $X \times X$ is a dense subspace of $Y \times Y$ and $d \colon X \times X \to \mathbb{R}$ is uniformly continuous, there exists a unique uniformly continuous map $D \colon Y \times Y \to \mathbb{R}$ with $D\vert_{X \times X} = d$. Due to $D$ being continuous, $\left. S \defeq \left\{ (x,y) \in Y^{2} \, \right| D(x,y) \geq 0 , \, D(x,y) = D(y,x) \right\}$ is closed in $Y^{2}$ and $\left. T \defeq \left\{ (x,y,z) \in Y^{3} \, \right| D(x,z) \leq D(x,y) + D(y,z) \right\}$ is closed in $Y^{3}$. Since $D\vert_{X \times X} = d$ is a metric, $X^{2} \subseteq S$ and $X^{3} \subseteq T$, and therefore $Y^{2} = S$ and $Y^{3} = T$ by density of $X$ in $Y$. Hence, $D$ is a uniformly continuous pseudo-metric on $Y$. It remains to prove that the uniformity of $Y$ is contained in the one generated by $D$, which will then imply that $D$ is a metric. To this end, let $U$ be an arbitrary entourage of $Y$. Choose a symmetric entourage $V$ of $Y$ such that $V \circ V \circ V \subseteq U$. As $d$ generates the uniformity of $X$, there exists $\epsilon > 0$ such that $\! \left. \left\{ (x,y) \in X^{2} \, \right| d(x,y) < \epsilon \right\} \subseteq V$. We will show that $U$ contains $W \defeq \left\{ (x,y) \in Y^{2} \left| \, D(x,y) < \tfrac{\epsilon}{3} \right\} \right.$. Let $(x,y) \in W$. As $D$ is continuous and $X$ is dense in~$Y$, we find $x' \in X \cap B_{D}\!\left(x,\tfrac{\epsilon}{3}\right)$ and $y' \in X \cap B_{D}\!\left(y,\tfrac{\epsilon}{3}\right)$ with $(x,x') \in V$ and $(y,y') \in V$. Then \begin{displaymath}
	d(x',y') \, = \, D(x',y') \, \leq \, D(x',x) + D(x,y) + D(y,y') \, < \, \epsilon ,
\end{displaymath} thus $(x',y') \in V$ and hence $(x,y) \in V \circ V \circ V \subseteq U$. Therefore, $W \subseteq U$ as desired. \end{proof}

\begin{lem}\label{lemma:extending.metrics} Let $G$ be a dense subgroup of a topological group $H$. If $d$ is a right-invariant compatible metric on $G$, then there exists a unique continuous map $D \colon H \times H \to \mathbb{R}$ with $D\vert_{G \times G} = d$, and furthermore $D$ is a right-invariant compatible metric on $H$. \end{lem}

\begin{proof} As $d$ is a right-invariant compatible metric on $G$, it generates the right uniformity of $G$, which coincides with the restriction of the right uniformity of $H$ to $G$. Thus, by Lemma~\ref{lemma:extending.uniform.metrics}, there exists a unique continuous map $D \colon H \times H \to \mathbb{R}$ with $D\vert_{G \times G} = d$, and moreover $D$ is a compatible metric on $H$. It remains to show that $D$ is right-invariant. Indeed, the continuity of $D$ implies that $\! \left. T \defeq \left\{ (x,y,z) \in H^{3} \, \right| D(xz,yz) = D(x,y) \right\}$ is closed in $H^{3}$. Since $G^{3} \subseteq T$ by right invariance of $d = D\vert_{G \times G}$ and $G$ is dense in $H$, it follows that $H^{3} = T$, as desired. \end{proof}

Our second preliminary note is the following variation on the well-known fact that quotients of Banach spaces by closed linear subspaces are again Banach spaces. We include a proof for the sake of convenience

\begin{lem}\label{lemma:density} Let $X$ and $Y$ be Banach spaces, and let $Y_{0}$ be any dense linear subspace of $Y$. If $T \colon X \to Y$ is a bounded linear operator such that \begin{displaymath}
	\forall y \in Y_{0} \colon \qquad \Vert y \Vert_{Y} = \inf \left\{ \Vert x \Vert_{X} \left| \, x \in T^{-1}(y) \right\} \! , \right.
\end{displaymath} then $T(X) = Y$. \end{lem}

\begin{proof} Let $y \in Y$. As $Y_{0}$ is dense in $Y$, there is a sequence $(z_{n})_{n \in \mathbb{N}}$ in $Y_{0}$ with $\Vert y - z_{n} \Vert_{Y} \leq 2^{-n}$ for each $n \in \mathbb{N}$. By assumption, there exists $x_{0} \in T^{-1}(z_{0})$ such that $\Vert x_{0} \Vert_{X} \leq \Vert z_{0} \Vert_{Y} + 1$. Likewise, our hypothesis asserts that, for each $n \in \mathbb{N}$, we find \begin{displaymath}
	x_{n+1} \in T^{-1}(z_{n+1} - z_{n})
\end{displaymath} with $\Vert x_{n+1} \Vert_{X} \leq \Vert z_{n+1} - z_{n} \Vert_{Y} + 2^{-(n+1)}$. For each $n \in \mathbb{N}$, consider $x^{\ast}_{n} \defeq \sum_{i\leq n} x_{i} \in X$ and note that $T(x^{\ast}_{n}) = \sum_{i\leq n} T(x_{i}) = z_{n}$. Furthermore, for all $m,n \in \mathbb{N}$ where $m>n$, \begin{displaymath}
	\Vert x^{\ast}_{m} - x^{\ast}_{n} \Vert_{X} \, \leq \, \sum_{i = n+1}^{m} \Vert x_{i} \Vert_{X} \, \leq \, \sum_{i=n+1}^{m} \Vert z_{i} - z_{i-1} \Vert_{Y} + 2^{-i} \, \leq \, 3 \sum_{i =n}^{m-1} 2^{-i} .
\end{displaymath} Hence, $(x^{\ast}_{n})_{n \in \mathbb{N}}$ is a Cauchy sequence in $X$, thus convergent to some point $x^{\ast} \in X$. Since $T$ is continuous, it follows that $T(x^{\ast}) = y$ as desired. \end{proof}

Now everything is in place for the proof of Theorem~\ref{theorem:equivariant.concentration}.

\begin{proof}[Proof of Theorem~\ref{theorem:equivariant.concentration}] Let us start off with a remark about the last assertion of Theorem~\ref{theorem:equivariant.concentration}: since $X = \spt \mu_{X}$ is compact, $\psi (X) = \psi (\spt \mu_{X}) = \spt (\psi_{\ast}\mu_{X})$ for any continuous mapping $\psi \colon X \to \mathrm{S}(G)$, whence the $G$-invariance of $\psi_{\ast}(\mu_{X})$ would imply the $G$-invariance of $\psi (X)$. We will establish the existence of the desired embedding by case analysis.
	
We first treat the precompact case. Assuming that $G$ is precompact, we find an embedding $h \colon G \to K$ into a compact group $K$ such that $K = \overline{h(G)}$, cf.~\cite[Corollary~3.7.17]{AT}. By Lemma~\ref{lemma:extending.metrics}, there is a unique continuous metric $d_{K} \colon K \times K \to \mathbb{R}$ such that $h \colon (G,d) \to (K,d_{K})$ is isometric, and furthermore $d_{K}$ is a compatible right-invariant metric on $K$. Let us denote by $\nu_{K}$ the normalized Haar measure on $K$. We prove that the sequence $\left(K_{n},d\!\!\upharpoonright_{K_{n}},\mu_{n}\!\!\upharpoonright_{K_{n}}\right)_{n \in \mathbb{N}}$ concentrates to $(K,d_{K},\nu_{K})$. For each $n \in \mathbb{N}$, the map $p_{n} \defeq h\vert_{K_{n}} \colon \left(K_{n},d\!\!\upharpoonright_{K_{n}}\right) \to (K,d_{K})$ is an isometric embedding, whence $\mathrm{Lip}_{1}(K_{n},d\!\!\upharpoonright_{K_{n}}) = {\mathrm{Lip}_{1}(K,d_{K})} \circ {p_{n}}$. According to Theorem~\ref{theorem:measure.concentration}, it now remains to show that $(p_{n})_{\ast}(\mu_{n}\!\!\upharpoonright_{K_{n}}) = h_{\ast}(\mu_{n}) \longrightarrow \nu_{K}$ weakly as $n \to \infty$. To this end, let $f \in \mathrm{C}(K) = \mathrm{RUCB}(K)$. By Corollary~\ref{corollary:wasserstein.convergence.to.inavriance} and Lemma~\ref{lemma:translated.convergence}(2), the sequence $(h_{\ast}(\mu_{n}))_{n \in \mathbb{N}}$ UEB-converges to invariance over $K$, in particular \begin{displaymath}
	\forall x \in K \colon \qquad \left\lvert \int f(xy) \, dh_{\ast}(\mu_{n})(y) - \int f(y) \, dh_{\ast}(\mu_{n})(y) \right\rvert \, \longrightarrow \, 0 \quad (n \to \infty) .
\end{displaymath} Due to Lebesgue's dominated convergence theorem, it follows that \begin{displaymath}
	\int \left\lvert \int f(xy) \, dh_{\ast}(\mu_{n})(y) - \int f(y) \, dh_{\ast}(\mu_{n})(y) \right\rvert \, d\nu_{K}(x) \, \longrightarrow \, 0 \quad (n \to \infty) .
\end{displaymath} Thanks to the right invariance of $\nu_{K}$ along with Fubini's theorem, we also have \begin{align*}
	\left\lvert \int f \, d\nu_{K} - \int f \, dh_{\ast}(\mu_{n}) \right\rvert \, &= \, \left\lvert \int \int f(xy) \, d\nu_{K}(x) \, dh_{\ast}(\mu_{n})(y) - \int f(y) \, dh_{\ast}(\mu_{n})(y) \right\rvert \\
	&= \, \left\lvert \int \int f(xy) \, dh_{\ast}(\mu_{n})(y) \, d\nu_{K}(x) - \int f(y) \, dh_{\ast}(\mu_{n})(y) \right\rvert \\
	&\leq \, \int \left\lvert \int f(xy) \, dh_{\ast}(\mu_{n})(y) - \int f(y) \, dh_{\ast}(\mu_{n})(y) \right\rvert \, d\nu_{K}(x)
\end{align*} for all $n \in \mathbb{N}$, which by the above implies that $\int f \, dh_{\ast}(\mu_{n}) \longrightarrow \int f \, d\nu_{K}$ as $n \to \infty$. This shows that $(p_{n})_{\ast}(\mu_{n}\!\!\upharpoonright_{K_{n}}) = h_{\ast}(\mu_{n}) \longrightarrow \nu_{K}$ weakly as $n \to \infty$, whence $\left(K_{n},d\!\!\upharpoonright_{K_{n}},\mu_{n}\!\!\upharpoonright_{K_{n}}\right)_{n \in \mathbb{N}}$ indeed concentrates to $(K,d_{K},\nu_{K})$. In view of (B), this necessitates that $d_{\mathrm{conc}}((K,d_{K}),(X,d_{X})) = 0$, wherefore the $mm$-spaces $(K,d_{K},\nu_{K})$ and $(X,d_{X},\mu_{X})$ are isomorphic~\cite[Theorem~5.16]{ShioyaBook}. Since both $\nu_{K}$ and $\mu_{X}$ have full support, this entails the existence of an isometric bijection $\zeta \colon (X,d_{X}) \to (K,d_{K})$ with $\zeta_{\ast}(\mu_{X}) = \nu_{K}$. Also, given that $h$ embeds the topological group $G$ densely into $K$, we obtain a $G$-equivariant homeomorphism $\chi \colon \mathrm{S}(G) \to \mathrm{S}(K)$ by setting \begin{displaymath}
	\chi (\xi)(f) \defeq \xi (f \circ h) \qquad (\xi \in \mathrm{S}(G), \, f \in \mathrm{C}(K)) ,
\end{displaymath} where $\mathrm{C}(K) = \mathrm{RUCB}(K)$ by compactness of $K$. As the $K$-equivariant map $\eta_{K} \colon K \to \mathrm{S}(K)$ is a homeomorphism due to Gelfand duality, $\phi \defeq \chi^{-1} \circ \eta_{K} \colon K \to \mathrm{S}(G)$ is a $G$-equivariant homeomorphism. Therefore, we conclude that $\psi \defeq \phi \circ \zeta \colon X \to \mathrm{S}(G)$ is a homeomorphism and that $\psi_{\ast}(\mu_{X}) = \phi_{\ast}(\nu_{K})$ is $G$-invariant. This settles the precompact case.
	
For the rest of the proof, let us assume that $G$ is not precompact. Let $U \in \mathcal{U}(G)$ be as in Corollary~\ref{corollary:usp.2}. Since $K_{n} \defeq \spt \mu_{n}$ is compact for every $n \in \mathbb{N}$, there exists $(g_{n})_{n \in \mathbb{N}} \in G^{\mathbb{N}}$ with $UK_{m}g_{m} \cap UK_{n}g_{n} = \emptyset$ for any two distinct $m,n \in \mathbb{N}$. For each $n \in \mathbb{N}$, consider the push-forward Borel probability measure $\nu_{n} \defeq \bigl(\rho_{g_{n}}\bigr)_{\ast}(\mu_{n})$ on $G$, and note that $S_{n} \defeq \spt \nu_{n} = K_{n}g_{n}$. We conclude that \begin{equation}\tag{i}\label{disjoint}
	\forall m,n \in \mathbb{N}, \, m \ne n \colon \qquad US_{m} \cap US_{n} = \emptyset .
\end{equation} Due to Corollary~\ref{corollary:wasserstein.convergence.to.inavriance} and Lemma~\ref{lemma:translated.convergence}(1), the sequence $(\nu_{n})_{n \in \mathbb{N}}$ UEB-converges to invariance over $G$. As the metric $d$ is right-invariant, \begin{displaymath} 
	(K_{n},d\!\!\upharpoonright_{K_{n}},\mu_{n}\!\!\upharpoonright_{K_{n}}) \, \longrightarrow \, (S_{n},d\!\!\upharpoonright_{S_{n}},\nu_{n}\!\!\upharpoonright_{S_{n}}), \qquad x \, \longmapsto \, xg_{n}
\end{displaymath} is an $mm$-space isomorphism for every $n \in \mathbb{N}$. Therefore, since $(K_{n},d\!\!\upharpoonright_{K_{n}},\mu_{n}\!\!\upharpoonright_{K_{n}})_{n \in \mathbb{N}}$ concentrates to $(X,d_{X},\mu_{X})$, so~does~$(S_{n},d\!\!\upharpoonright_{S_{n}},\nu_{n}\!\!\upharpoonright_{S_{n}})_{n \in \mathbb{N}}$. Consider the Prokhorov distance $(d_{X})_{\mathrm{P}}$ on $\mathrm{P}(X)$ associated with the metric $d_{X}$ (see Section~\ref{section:gromov}). Due to Theorem~\ref{theorem:measure.concentration}, there exists a sequence of Borel maps $p_{n} \colon S_{n} \to X$ ($n \in \mathbb{N}$) such that \begin{enumerate}
	\item[$(1)$] $(d_{X})_{\mathrm{P}}((p_{n})_{\ast}(\overline{\nu}_{n}),\mu_{X}) \longrightarrow 0$ as $n \to 0$, \smallskip
	\item[$(2)$] $(\mathrm{me}_{\overline{\nu}_{n}})_{\mathrm{H}}(\mathrm{Lip}_{1}(X,d_{X}) \circ p_{n},\mathrm{Lip}_{1}(S_{n},d_{n})) \longrightarrow 0$ as $n \to \infty$,
\end{enumerate} where $\overline{\nu}_{n} \defeq \nu_{n}\!\!\upharpoonright_{S_{n}}$ and $d_{n} \defeq d\!\!\upharpoonright_{S_{n}}$ for $n \in \mathbb{N}$. We show that, for every $f \in \mathrm{Lip}^{\infty}(G,d)$, \begin{displaymath}
	\left. T(f) \defeq \left\{ (f_{n})_{n \in \mathbb{N}} \in \mathrm{Lip}_{\ell}^{\ell}(X,d_{X})^{\mathbb{N}} \, \right| \ell \geq 0, \ \lim\nolimits_{n \to \infty} \mathrm{me}_{\overline{\nu}_{n}}(f_{n} \circ p_{n},f\vert_{S_{n}}) = 0 \right\} .
\end{displaymath} is a non-empty set. For this purpose, let $f \in \mathrm{Lip}_{\ell}^{\ell}(G,d)$ for some $\ell \geq 1$. Thanks to~$(2)$, there exists a sequence $(f_{n})_{n \in \mathbb{N}} \in \mathrm{Lip}_{1}(X,d_{X})^{\mathbb{N}}$ such that $\mathrm{me}_{\overline{\nu}_{n}}(f_{n} \circ p_{n}, \ell^{-1}f\vert_{S_{n}}) \longrightarrow 0$ as $n \to \infty$. For each $n \in \mathbb{N}$, it follows that \begin{displaymath}
	f'_{n} \defeq ((\ell f_{n}) \wedge \ell ) \vee (- \ell) \in \mathrm{Lip}_{\ell}^{\ell}(X,d_{X})
\end{displaymath} and moreover $\mathrm{me}_{\overline{\nu}_{n}}(f'_{n} \circ p_{n},f\vert_{S_{n}}) \, \leq \, \mathrm{me}_{\overline{\nu}_{n}}((\ell f_{n}) \circ p_{n},f\vert_{S_{n}}) \, \leq \, \ell \mathrm{me}_{\overline{\nu}_{n}}(f_{n} \circ p_{n}, \ell^{-1}f\vert_{S_{n}})$, which shows that $(f'_{n})_{n \in \mathbb{N}} \in T(f)$, as desired.

Next let us prove that \begin{equation}\tag{ii}\label{confluence}
	\forall f \in \mathrm{Lip}^{\infty}(G,d) \ \forall (f_{n})_{n \in \mathbb{N}}, (f'_{n})_{n \in \mathbb{N}} \in T(f) \colon \quad \lim\nolimits_{n \to \infty} \Vert f_{n} - f'_{n} \Vert_{\infty} = 0 .
\end{equation} To this end, let $f \in \mathrm{Lip}^{\infty}(G,d)$ and $(f_{n})_{n \in \mathbb{N}}, (f'_{n})_{n \in \mathbb{N}} \in T(f)$. Fix $\ell \geq 1$ so that \begin{displaymath}
	\{ f_{n} \mid n \in \mathbb{N} \} \cup \{ f_{n}' \mid n \in \mathbb{N} \} \subseteq \mathrm{Lip}_{\ell}^{\ell}(X,d_{X}) .
\end{displaymath} According to~(1), there exists a sequence $(\delta_{n})_{n \in \mathbb{N}}$ of positive real numbers converging to $0$ such that $(d_{X})_{\mathrm{P}}((p_{n})_{\ast}(\overline{\nu}_{n}),\mu_{X}) < \delta_{n}$ for all $n \in \mathbb{N}$. Let $n \in \mathbb{N}$. Consider $\sigma_{n} \defeq \mathrm{me}_{\overline{\nu}_{n}}(f_{n} \circ p_{n},f\vert_{S_{n}})$ and $\tau_{n} \defeq \mathrm{me}_{\overline{\nu}_{n}}(f'_{n} \circ p_{n},f\vert_{S_{n}})$. Note that \begin{displaymath}
	B_{n} \defeq \{ x \in X \mid \vert f_{n}(x) - f'_{n}(x) \vert \leq \sigma_{n} + \tau_{n} + \ell \delta_{n} \}
\end{displaymath} is a Borel subset of $X$ containing $B_{d_{X}}(C_{n},\delta_{n})$ for $C_{n} \defeq \{ x \in X \mid \vert f_{n}(x) - f'_{n}(x) \vert \leq \sigma_{n} + \tau_{n} \}$. Considering the Borel sets \begin{align*}
	& D_{n} \defeq \{ s \in S_{n} \mid \vert f_{n}(p_{n}(s)) - f(s) \vert \leq \sigma_{n} \} , & D'_{n} \defeq \{ s \in S_{n} \mid \vert f'_{n}(p_{n}(s)) - f(s) \vert \leq \tau_{n} \} ,
\end{align*} we observe that $D_{n} \cap D'_{n} \subseteq p_{n}^{-1}(C_{n})$, and therefore \begin{displaymath}
	\overline{\nu}_{n}(p_{n}^{-1}(C_{n})) \, \geq \, \overline{\nu}_{n}(D_{n} \cap D'_{n}) \, \geq \, 1 - \overline{\nu}_{n}(S_{n}\setminus D_{n}) - \overline{\nu}_{n}(S_{n}\setminus D'_{n}) \, \geq \, 1 - \sigma_{n} - \tau_{n} .
\end{displaymath} It follows that\begin{displaymath}
	\mu_{X}(B_{n}) \, \geq \, \mu_{X}(B_{d_{X}}(C_{n},\delta_{n})) \, \geq \, \overline{\nu}_{n}(p_{n}^{-1}(C_{n})) - \delta_{n} \, \geq \, 1 - \sigma_{n} - \tau_{n} - \delta_{n} .
\end{displaymath} This shows that $\mathrm{me}_{\mu_{X}}(\vert f_{n} - f'_{n} \vert,\mathbf{0}) \leq \sigma_{n} + \tau_{n} + \ell \delta_{n}$. Since this is true for arbitrary $n \in \mathbb{N}$, the definition of $T(f)$ and our choice of $(\delta_{n})_{n \in \mathbb{N}}$ imply that $\mathrm{me}_{\mu_{X}}(\vert f_{n} - f'_{n} \vert,\mathbf{0}) \longrightarrow 0$ as $n \to \infty$, i.e., $\vert f_{n} - f'_{n} \vert \longrightarrow \mathbf{0}$ in the measure $\mu_{X}$ as $n \to \infty$. Let $k \defeq 2\ell$. As~$\spt \mu_{X} = X$, the restriction of $\mathrm{me}_{\mu_{X}}$ to $\mathrm{C}(X)$ is a metric, wherefore the induced topology on $\mathrm{C}(X)$, i.e., the topology of convergence in $\mu_{X}$, is Hausdorff. By the Arzel\`{a}-Ascoli theorem, $\mathrm{Lip}_{k}^{k}(X,d_{X})$ is compact with respect to the topology of uniform convergence. Given that the latter contains the topology of convergence in $\mu_{X}$, the two topologies coincide on the set $\mathrm{Lip}_{k}^{k}(X,d_{X})$. Since \begin{displaymath}
	\{ \vert f_{n} - f'_{n} \vert \mid n \in \mathbb{N} \} \cup \{ \mathbf{0} \} \subseteq \mathrm{Lip}_{k}^{k}(X,d_{X}) ,
\end{displaymath} we conclude that $\vert f_{n} - f'_{n} \vert \longrightarrow \mathbf{0}$ uniformly as $n \to \infty$. This proves~\eqref{confluence}.

Fix a non-principal ultrafilter $\mathcal{F}$ on $\mathbb{N}$. Appealing to~\eqref{confluence} and the Arzel\`{a}-Ascoli theorem again, let us define \begin{displaymath}
	\Phi \colon \, \mathrm{Lip}^{\infty}(G,d) \, \longrightarrow \, \mathrm{Lip}(X,d_{X})
\end{displaymath} by setting \begin{displaymath}
	\Phi (f) \defeq \lim\nolimits_{n \to \mathcal{F}} f_{n} \qquad \left(f \in \mathrm{Lip}^{\infty}(G,d), \, (f_{n})_{n \in \mathbb{N}} \in T(f) \right) ,
\end{displaymath} where the ultrafilter convergence applies to the uniform topology. We will show that $\Phi$ is a homomorphism of unital $\mathbb{R}$-algebras. Evidently, $(\boldsymbol{r})_{n \in \mathbb{N}} \in T(\boldsymbol{r})$ and thus $\Phi (\boldsymbol{r}) = \lim_{n \to \mathcal{F}} \boldsymbol{r} = \boldsymbol{r}$ for any $r \in \mathbb{R}$. Let $f,f' \in \mathrm{Lip}^{\infty}(G,d)$. Pick any $(f_{n})_{n \in \mathbb{N}} \in T(f)$ and $(f'_{n})_{n \in \mathbb{N}} \in T(f')$, and choose $\ell \geq 1$ with $\{ f_{n} \mid n\in \mathbb{N} \} \cup \{ f'_{n} \mid n \in \mathbb{N} \} \subseteq \mathrm{Lip}_{\ell}^{\ell}(X,d_{X})$ and \begin{displaymath}
	\max\left\{ \Vert f \Vert_{\infty}, \, \Vert f' \Vert_{\infty} \right\} \, \leq \, \ell .
\end{displaymath} It is easy to check that $f_{n} + f'_{n} \in \mathrm{Lip}_{2\ell}^{2\ell}(X,d_{X})$ and $f_{n}f'_{n} \in \mathrm{Lip}_{2\ell^{2}}^{\ell^{2}}(X,d_{X})$ for every $n \in \mathbb{N}$. Let $\sigma_{n} \defeq \mathrm{me}_{\overline{\nu}_{n}}(f_{n} \circ p_{n},f\vert_{S_{n}})$ and $\tau_{n} \defeq \mathrm{me}_{\overline{\nu}_{n}}(f'_{n} \circ p_{n},f'\vert_{S_{n}})$ for $n \in \mathbb{N}$. For each $n \in \mathbb{N}$, \begin{align*}
	\{ s \in S_{n} \mid \, &\vert (f_{n}+f'_{n})(p_{n}(s)) - (f+f')(s) \vert > \sigma_{n} + \tau_{n} \} \\
	&\subseteq \{ s \in S_{n} \mid \vert f_{n}(p_{n}(s)) - f(s) \vert > \sigma_{n} \} \cup \{ s \in S_{n} \mid \vert f'_{n}(p_{n}(s)) - f'(s) \vert > \tau_{n} \}
\end{align*} and thus \begin{displaymath}
	\mathrm{me}_{\overline{\nu}_{n}}((f_{n}+f'_{n}) \circ p_{n},(f+f')\vert_{S_{n}}) \, \leq \, \sigma_{n} + \tau_{n} ,
\end{displaymath} as well as \begin{align*}
	\{ s \in S_{n} \mid \, &\vert (f_{n}f'_{n})(p_{n}(s)) - (ff')(s) \vert > \ell (\sigma_{n} + \tau_{n}) \} \\
	&\subseteq \{ s \in S_{n} \mid \vert f_{n}(p_{n}(s)) - f(s) \vert > \sigma_{n} \} \cup \{ s \in S_{n} \mid \vert f'_{n}(p_{n}(s)) - f'(s) \vert > \tau_{n} \} 
\end{align*} and therefore \begin{displaymath}
	\mathrm{me}_{\overline{\nu}_{n}}((f_{n}f'_{n}) \circ p_{n},(ff')\vert_{S_{n}}) \, \leq \, \max \{ \ell(\sigma_{n} + \tau_{n}), \sigma_{n} + \tau_{n} \} \, \leq \, \ell(\sigma_{n} + \tau_{n}) .
\end{displaymath} Hence, $(f_{n}+f'_{n})_{n \in \mathbb{N}} \in T(f+f')$ and $(f_{n}f'_{n})_{n \in \mathbb{N}} \in T(ff')$, which readily implies that \begin{displaymath}
	\Phi (f+f') \, = \, \lim\nolimits_{n \to \mathcal{F}} f_{n}+f'_{n} \, = \, \left( \lim\nolimits_{n\to \mathcal{F}} f_{n} \right) + \left( \lim\nolimits_{n\to \mathcal{F}} f'_{n} \right) \, = \, \Phi (f) + \Phi(f')
\end{displaymath} and likewise \begin{displaymath}
	\Phi (ff') \, = \, \lim\nolimits_{n \to \mathcal{F}} f_{n}f'_{n} \, = \, \left( \lim\nolimits_{n\to \mathcal{F}} f_{n} \right) \left( \lim\nolimits_{n\to \mathcal{F}} f'_{n} \right) \, = \, \Phi (f)\Phi(f') .
\end{displaymath} This shows that $\Phi$ is indeed a homomorphism between unital $\mathbb{R}$-algebras.

Next let us prove that \begin{equation}\tag{iii}\label{quotient}
	\forall f \in \mathrm{Lip}(X,d_{X}) \colon \qquad \Vert f \Vert_{\infty} \, = \, \min \{ \Vert f^{\ast} \Vert_{\infty} \mid f^{\ast} \in \Phi^{-1}(f) \} .
\end{equation} Note that this will imply that $\Phi$ is surjective. We start our proof of~\eqref{quotient} by observing that $\Phi$ is contractive with respect to the supremum norm, i.e., \begin{displaymath}
	\forall f \in \mathrm{Lip}^{\infty}(G,d) \colon \qquad \Vert \Phi(f) \Vert_{\infty} \, \leq \, \Vert f \Vert_{\infty} .
\end{displaymath} Indeed, if $f \in \mathrm{Lip}^{\infty}(G,d)$ and $(f_{n})_{n \in \mathbb{N}} \in T(f)$, then it is easily checked that \begin{displaymath}
	((f_{n} \wedge \Vert f \Vert_{\infty}) \vee (-\Vert f \Vert_{\infty}))_{n \in \mathbb{N}} \, \in \, T(f) ,
\end{displaymath} and thus \begin{displaymath}
	\Phi(f) \, = \, \lim\nolimits_{n \to \mathcal{F}} \left( (f_{n} \wedge \Vert f \Vert_{\infty}) \vee (-\Vert f \Vert_{\infty}) \right) ,
\end{displaymath} whence $\Vert \Phi (f) \Vert_{\infty} \leq \Vert f \Vert_{\infty}$ as desired. Furthermore, since $d$ generates the topology of $G$, there exists $\epsilon > 0$ such that $B_{d}(e,\epsilon) \subseteq U$. By right-invariance of $d$, it follows that \begin{displaymath}
	B_{d}(S_{m},\epsilon) \cap B_{d}(S_{n},\epsilon ) \subseteq US_{m} \cap US_{n} = \emptyset
\end{displaymath} for any two distinct $m,n \in \mathbb{N}$. To prove~\eqref{quotient}, let $\ell \geq 1$ and $f \in \mathrm{Lip}_{\ell}(X,d_{X})$. Put~$c \defeq \Vert f \Vert_{\infty}$. According to~$(2)$, there exists a sequence of functions $f_{n} \in \mathrm{Lip}_{1}(S_{n},d_{n})$ $(n \in \mathbb{N})$ such that $\mathrm{me}_{\overline{\nu}_{n}}((\ell^{-1}f) \circ p_{n}, f_{n}) \longrightarrow 0$ as $n \to \infty$. For every $n \in \mathbb{N}$, \begin{displaymath}
	f'_{n} \, \defeq \, ((\ell f_{n}) \wedge c ) \vee (- c) \, \in \, \mathrm{Lip}_{\ell}^{c}(S_{n},d_{n})
\end{displaymath} and \begin{displaymath}
	\mathrm{me}_{\overline{\nu}_{n}}(f \circ p_{n},f'_{n}) \, \leq \, \mathrm{me}_{\overline{\nu}_{n}}(f \circ p_{n},\ell f_{n})  \, \leq \, \ell \mathrm{me}_{\overline{\nu}_{n}}((\ell^{-1} f) \circ p_{n},f_{n}) .
\end{displaymath} Consider the set $S \defeq \bigcup_{n \in \mathbb{N}} S_{n}$ and define $f' \colon S \to \mathbb{R}$ by setting $f'\vert_{S_{n}} = f'_{n}$ for every $n \in \mathbb{N}$. Then $f' \in \mathrm{Lip}^{c}_{k}(S,d\!\!\upharpoonright_{S})$ for $k \defeq \max\{ \ell, 2c \epsilon^{-1}\}$, since $B_{d}(S_{m},\epsilon) \cap B_{d}(S_{n},\epsilon ) = \emptyset$ for any two distinct $m,n \in \mathbb{N}$. Utilizing a standard construction, we define $f^{\ast} \colon G \to \mathbb{R}$ by \begin{displaymath}
	f^{\ast}(g) \defeq \left( \left( \inf\nolimits_{s \in S} f'(s) + k d(g,s) \right) \wedge c \right) \vee (-c) \qquad (g \in G)
\end{displaymath} and observe that $f^{\ast} \in \mathrm{Lip}^{c}_{k}(G,d)$. Since moreover $f^{\ast}\vert_{S} = f'$, it follows that \begin{displaymath}
	\mathrm{me}_{\overline{\nu}_{n}}(f \circ p_{n},f^{\ast}\vert_{S_{n}}) \, = \, \mathrm{me}_{\overline{\nu}_{n}}(f \circ p_{n},f'_{n}) \leq \ell \mathrm{me}_{\overline{\nu}_{n}}((\ell^{-1} f) \circ p_{n},f_{n}) 
\end{displaymath} for all $n \in \mathbb{N}$. Hence, $(f)_{n \in \mathbb{N}} \in T(f^{\ast})$ and therefore $\Phi (f^{\ast}) = \lim_{n\to \mathcal{F}} f = f$. Finally, \begin{displaymath}
	\Vert f \Vert_{\infty} = \Vert \Phi(f^{\ast}) \Vert_{\infty} \leq \Vert f^{\ast} \Vert_{\infty} \leq c = \Vert f \Vert_{\infty} 
\end{displaymath} and thus $\Vert f^{\ast} \Vert_{\infty} = \Vert f \Vert_{\infty}$, as desired.

Let us consider the unique continuous linear operator $\overline{\Phi} \colon \mathrm{RUCB}(G) \to \mathrm{C}(X)$ extending~$\Phi$. Since $\Phi$ is a homomorphism of unital $\mathbb{R}$-algebras, so is $\overline{\Phi}$. Moreover, $\overline{\Phi}$ is surjective by~\eqref{quotient}, Lemma~\ref{lemma:density}, and the density of $\mathrm{Lip}(X,d_{X})$ in $\mathrm{C}(X)$. The map \begin{displaymath}
	\nu \colon \, \mathrm{RUCB}(G) \, \longrightarrow \, \mathbb{R} , \qquad f \, \longmapsto \, \lim\nolimits_{n \to \mathcal{F}} \int f \, d\nu_{n} ,
\end{displaymath} is a left-invariant mean, cf.~\cite[Proof of Theorem~3.2]{SchneiderThom}. We will show that \begin{equation}\tag{iv}\label{convergence}
	\forall f \in \mathrm{RUCB}(G) \colon \qquad \int \overline{\Phi}(f) \,d\mu_{X} = \nu (f) .
\end{equation} Since both $\nu$ and $\mathrm{RUCB}(G) \to \mathbb{R}, \, f \mapsto \int \overline{\Phi}(f) \,d\mu_{X}$ are continuous linear maps and $\mathrm{Lip}^{\infty}(G,d)$ is a dense linear subspace of $\mathrm{RUCB}(G)$, it suffices to prove that \begin{displaymath}
	\forall f \in \mathrm{Lip}^{\infty}(G,d) \colon \qquad \int \Phi(f) \,d\mu_{X} = \nu (f) .
\end{displaymath} Let $f \in \mathrm{Lip}^{\infty}(G,d)$ and $(f_{n})_{n \in \mathbb{N}} \in T(f)$. Then $c \defeq \Vert f \Vert_{\infty} \vee \sup_{n \in \mathbb{N}} \Vert f_{n} \Vert_{\infty} < \infty$. Since $(d_{X})_{\mathrm{P}}$ metrizes the weak topology on $\mathrm{P}(X)$, assertion~$(1)$ implies that \begin{displaymath}
	\epsilon_{n} \defeq \left\lvert \int \Phi (f) \, d\mu_{X} - \int \Phi (f) \, d(p_{n})_{\ast}(\overline{\nu}_{n}) \right\rvert \, \longrightarrow \, 0 \qquad (n \to \infty) .
\end{displaymath} Considering that $\mathcal{F}$ is non-principal and, moreover, \begin{align*}
	\left\lvert \int \Phi(f) \, d\mu_{X} - \int f \,d\nu_{n} \right\rvert \, &\leq \, \left\lvert \int \Phi(f) \, d\mu_{X} - \int \Phi (f) \,d(p_{n})_{\ast}(\overline{\nu}_{n}) \right\rvert \\
		& \qquad \qquad \qquad \qquad + \left\lvert \int \Phi(f) \, d(p_{n})_{\ast}(\overline{\nu}_{n}) - \int f_{n} \,d(p_{n})_{\ast}(\overline{\nu}_{n}) \right\rvert \\
		& \qquad \qquad \qquad \qquad + \left\lvert \int f_{n} \circ p_{n} \, d\overline{\nu}_{n} - \int f\vert_{S_{n}} \,d\overline{\nu}_{n} \right\rvert \\
		& \leq \, \epsilon_{n} + \Vert \Phi (f) - f_{n} \Vert_{\infty} + (1+2c)\mathrm{me}_{\overline{\nu}_{n}}(f_{n} \circ p_{n},f\vert_{S_{n}})
\end{align*} for all $n \in \mathbb{N}$, we conclude that $\left\lvert \int \Phi(f) \, d\mu_{X} - \int f \,d\nu_{n} \right\rvert \longrightarrow 0$ as $n \to \mathcal{F}$. This proves~\eqref{convergence}.

Finally, let us consider the continuous map $\psi \colon X \to \mathrm{S}(G)$ given by \begin{displaymath}
	\psi(x)(f) \defeq \overline{\Phi} (f) (x) \qquad (x \in X, \, f \in \mathrm{RUCB}(G)) .
\end{displaymath} Since $\overline{\Phi}$ is onto, $\psi$ is a topological embedding. To see that $\psi_{\ast}(\mu_{X})$ is $G$-invariant, let us note the following: for every $f \in \mathrm{C}(\mathrm{S}(G))$, since $f(\xi) = \xi (f \circ \eta_{G})$ for all $\xi \in \mathrm{S}(G)$, we have \begin{displaymath}
	f(\psi (x)) = \psi(x) (f \circ \eta_{G}) = \overline{\Phi} (f \circ \eta_{G})(x)
\end{displaymath} for all $x \in X$, i.e., $f \circ \psi = \overline{\Phi} (f \circ \eta_{G})$. Also, being a Borel probability measure on a metrizable compact space, $\mu_{X}$ is regular. As $\psi$ is a continuous map between compact Hausdorff spaces, $\psi_{\ast}(\mu_{X})$ must be regular as well. Therefore, in order to establish the $G$-invariance of $\psi_{\ast}(\mu_{X})$, it suffices to observe that, for all $f \in \mathrm{C}(\mathrm{S}(G))$ and $g \in G$, \begin{align*}
	\int & f \circ \tau_{g} \, d\psi_{\ast}(\mu_{X}) \ = \ \int f \circ \tau_{g} \circ \psi \, d\mu_{X} \ = \ \int \overline{\Phi}(f \circ \tau_{g} \circ \eta_{G}) \, d\mu_{X} \ = \ \int \overline{\Phi} (f \circ \eta_{G} \circ \lambda_{g}) \, d\mu_{X} \\
	&\stackrel{\eqref{convergence}}{=} \, \nu(f \circ \eta_{G} \circ \lambda_{g}) \, = \, \nu(f \circ \eta_{G}) \, \stackrel{\eqref{convergence}}{=} \, \int \overline{\Phi}(f \circ \eta_{G}) \, d\mu_{X} \ = \ \int f \circ \psi \, d\mu_{X} \ = \ \int f \, d\psi_{\ast}(\mu_{X}) ,
\end{align*} where $\tau_{g} \colon \mathrm{S}(G) \to \mathrm{S}(G), \, \xi \mapsto g\xi$. This completes the proof. \end{proof}

\section{Observable diameters}\label{section:observable.diameters}

In this section we will prove Theorem~\ref{theorem:observable.diameters} and then deduce Corollary~\ref{corollary:observable.diameters}. For a start, let us briefly recall Gromov's concept of observable diameters~\cite[Chapter~3$\tfrac{1}{2}$]{Gromov99}. For further reading, we refer to~\cite{Gromov99,ledoux,ShioyaBook}.

\begin{definition}\label{definition:observable.diameters} Let $\alpha > 0$. The \emph{$\alpha$-partial diameter} of a Borel probability measure $\nu$ on $\mathbb{R}$ is \begin{displaymath}
	\mathrm{PartDiam}(\nu,1-\alpha) \defeq \inf \{ \diam (B,d_{\mathbb{R}}) \mid B\subseteq \mathbb{R} \textnormal{ Borel, } \nu(B) \geq 1 - \alpha \} ,
\end{displaymath} where $d_{\mathbb{R}}$ denotes the Euclidean metric on $\mathbb{R}$. Given any Borel probability measure $\mu$ on a topological space $X$ and a continuous pseudo-metric~$d$ on $X$, we refer to the quantity \begin{displaymath}
	\mathrm{ObsDiam}(X,d,\mu;-\alpha) \defeq \sup \{ \mathrm{PartDiam}(f_{\ast}(\mu),1-\alpha) \mid f \in \mathrm{Lip}_{1}(X,d) \} .
\end{displaymath} as the corresponding \emph{$\alpha$-observable diameter}. \end{definition}

\begin{remark}\label{remark:monotonicity} Let $\alpha > 0$. If $\mu$ is a Borel probability measure on a topological space $X$ and $d_{0} \leq d_{1}$ are continuous pseudo-metrics on $X$, then \begin{displaymath}
	\mathrm{ObsDiam}(X,d_{0},\mu;-\alpha) \, \leq \, \mathrm{ObsDiam}(X,d_{1},\mu;-\alpha) .
\end{displaymath} In particular, if $G$ is a topological group acting continuously on a compact Hausdorff space~$X$ and $d$ is a continuous pseudo-metric on $X$, then \begin{displaymath}
	\sup\nolimits_{x \in X} \mathrm{ObsDiam}(G,d_{G,x},\mu; -\alpha) \, \leq \, \mathrm{ObsDiam}(G,d_{G,X},\mu; -\alpha)
\end{displaymath} for every Borel probability measure $\mu$ on $G$. \end{remark}

\begin{proof}[Proof of Theorem~\ref{theorem:observable.diameters}] Let $X$ be a $G$-flow equipped with a $G$-invariant regular Borel probability measure $\nu$. Consider a continuous pseudo-metric $d$ on $X$ and let \begin{displaymath}
	D \defeq \sup\nolimits_{\alpha > 0} \liminf\nolimits_{i \to I} \sup\nolimits_{x \in X} \mathrm{ObsDiam}\! \left(G,d_{G,x},\mu_{i}; -\alpha\right) .
\end{displaymath} Let $E \subseteq G$ be finite and let $\epsilon > 0$. We show that \begin{displaymath}
	\int \sup\nolimits_{g \in E} d(x,gx) \, d\nu (x) \, \leq \, D + \epsilon .
\end{displaymath} To this end, let $U \defeq B_{d_{G,X}}\!\left(e,\tfrac{\epsilon}{8}\right)$ and pick a right-uniformly continuous function $p \colon G \to [0,1]$ with $p(e) = 1$ and $p(x) = 0$ for all $x \in G \setminus U$. For any $S \subseteq G$, define $p_{S} \colon G \to [0,1]$ by \begin{displaymath}
	p_{S}(x) \defeq \sup\nolimits_{s \in S} p\!\left( xs^{-1} \right) \qquad (x \in G) .
\end{displaymath} It is easy to see that $\{ p_{S} \mid S \subseteq G \}$ belongs to $\mathrm{RUEB}(G)$. Let $\delta \defeq \diam (X,d) + 1$. Since the net $(\mu_{i})_{i \in I}$ UEB-converges to invariance over $G$, we find $i_{0} \in I$ such that \begin{equation}\tag{$\ast$}\label{invariance}
	\forall i \in I, \ i \geq i_{0} \ \forall g \in E \colon \quad \sup\nolimits_{S \subseteq G} \left\lvert \int p_{S} \, d\mu_{i} - \int p_{S} \circ \lambda_{g} \, d\mu_{i} \right\rvert \, \leq \, \tfrac{\epsilon}{8\delta (\vert E \vert + 1)} .
\end{equation} Since $X$ is compact, there exists a finite non-empty subset $F \subseteq \mathrm{Lip}_{1}(X,d)$ such that \begin{displaymath}
	\forall x,y \in X \colon \qquad d(x,y) \, \leq \, \sup\nolimits_{f \in F} \vert f(x) - f(y) \vert + \tfrac{\epsilon}{2} ,
\end{displaymath} cf.~\cite[Exercise~7.4.13]{PestovBook}. For the rest of the proof, fix any $i \in I$ such that $i \geq i_{0}$ and \begin{displaymath}
	\sup\nolimits_{x \in X} \mathrm{ObsDiam}\!\left(G,d_{G,x},\mu_{i};-\tfrac{\epsilon}{8\delta \vert F \vert (\vert E \vert + 1)}\right) \, \leq \, D + \tfrac{\epsilon}{8} .
\end{displaymath} Let us prove that \begin{equation}\tag{$\ast \ast$}\label{aim}
	\sup\nolimits_{x \in X} \int \sup\nolimits_{g \in E} \sup\nolimits_{f \in F} \vert f(hx) - f(ghx) \vert \, d\mu_{i}(h) \, \leq \, D + \tfrac{\epsilon}{2} .
\end{equation} Let $x \in X$. Since for each $f \in F$ the map $f_{x} \colon G \to \mathbb{R} , \, h \mapsto f(hx)$ belongs to $\mathrm{Lip}_{1}(G,d_{G,x})$, our choice of $i$ ensures that \begin{displaymath}
	\sup\nolimits_{f \in F} \mathrm{PartDiam} \! \left( (f_{x})_{\ast}(\mu_{i}), 1-\tfrac{\epsilon}{8\delta \vert F \vert (\vert E \vert + 1)} \right) \, \leq \, D + \tfrac{\epsilon}{8} .
\end{displaymath} Hence, for each $f \in F$ there exists a Borel set $B_{f} \subseteq G$ such that $\mu_{i}(B_{f}) \geq 1 - \tfrac{\epsilon}{8\delta \vert F \vert (\vert E \vert + 1)}$ and $\diam (f_{x}(B_{f}),d_{\mathbb{R}}) \leq D + \tfrac{\epsilon}{8}$. Considering the Borel set $B \defeq \bigcap_{f \in F} B_{f}$, we deduce that $\mu_{i}(B) \geq 1 - \tfrac{\epsilon}{8\delta(\vert E \vert + 1)}$ and moreover $\diam (f_{x}(B),d_{\mathbb{R}}) \leq D + \tfrac{\epsilon}{8}$ for each $f \in F$. The former implies that $\int p_{B} \, d\mu_{i} \geq  1 -\tfrac{\epsilon}{8\delta(\vert E \vert + 1)}$. Thus, \eqref{invariance} asserts that $\int p_{B} \circ \lambda_{g} \, d\mu_{i} \geq 1 -\tfrac{\epsilon}{4\delta (\vert E \vert + 1)}$ for each $g \in G$. Since $UB = B_{d_{G,X}}\!\left(B,\tfrac{\epsilon}{8}\right)$ by right invariance of $d_{G,X}$, it readily follows that $\mu_{i}\! \left( g^{-1}B_{d_{G,X}}\!\left(B,\tfrac{\epsilon}{8}\right) \right) \geq 1 -\tfrac{\epsilon}{4\delta (\vert E \vert + 1)}$. Considering the Borel set \begin{displaymath}
	C \defeq B \cap \bigcap\nolimits_{g \in E} g^{-1}B_{d_{G,X}}\!\left(B,\tfrac{\epsilon}{8}\right) \! ,
\end{displaymath} we now conclude that $\mu_{i}(C) \geq 1-\tfrac{\epsilon}{4\delta}$. Furthermore, \begin{displaymath}
	\sup\nolimits_{g \in E} \sup\nolimits_{f \in F} \vert f(hx) - f(ghx) \vert \, \leq \, D + \tfrac{\epsilon}{4} .
\end{displaymath} for all $h \in C$. To see this, let $h \in C$. For each $g \in E$, there is $h_{g} \in B$ with $d_{G,X}(h_{g},gh) \leq \tfrac{\epsilon}{8}$. Hence, \begin{align*}
	\vert f(hx) - f(ghx) \vert \, &\leq \, \vert f(hx) - f(h_{g}x) \vert + \vert f(h_{g}x) - f(ghx) \vert \\
	&\leq \, \left(D + \tfrac{\epsilon}{8}\right) + d_{G,X}(h_{g},gh) \, \leq \, D + \tfrac{\epsilon}{4} 
\end{align*} for all $g \in E$ and $f \in F$, as desired. Consequently, since $F$ is contained in $\mathrm{Lip}_{1}(X,d)$, \begin{align*}
	\int \sup\nolimits_{g \in E} \sup\nolimits&_{f \in F} \vert f(hx) - f(ghx) \vert \, d\mu_{i}(h) \\ 
	&= \, \int_{C} \sup\nolimits_{g \in E} \sup\nolimits_{f \in F} \vert f(hx) - f(ghx) \vert \, d\mu_{i}(h) + \delta \mu_{i}(G\setminus C) \\
	&\leq \, \left( D + \tfrac{\epsilon}{4} \right) + \tfrac{\epsilon}{4} \, \leq \, D + \tfrac{\epsilon}{2} .
\end{align*} This proves~\eqref{aim}. Using the $G$-invariance of $\nu$ along with Fubini's theorem, we deduce that \begin{align*}
	\int \sup\nolimits_{g \in E} \sup\nolimits_{f \in F} \vert & f(x) - f(gx) \vert \, d\nu (x) \, = \, \int \int \sup\nolimits_{g \in E} \sup\nolimits_{f \in F} \vert f(hx) - f(ghx) \vert \, d\nu (x) \, d\mu_{i}(h) \\
	&= \, \int \int \sup\nolimits_{g \in E} \sup\nolimits_{f \in F} \vert f(hx) - f(ghx) \vert\, d\mu_{i}(h) \, d\nu (x) \, \stackrel{\eqref{aim}}{\leq} \, D + \tfrac{\epsilon}{2} 
\end{align*} and therefore \begin{displaymath}
	\int  \sup\nolimits_{g \in E} d(x,gx) \, d\nu(x) \, \leq \, \int  \sup\nolimits_{g \in E} \sup\nolimits_{f \in F} \vert f(x) - f(gx) \vert \, d\nu (x) + \tfrac{\epsilon}{2} \, \leq \, D + \epsilon .
\end{displaymath} This completes the proof. \end{proof}

\begin{proof}[Proof of Corollary~\ref{corollary:observable.diameters}] Let $X$ be a $G$-flow. Fix a continuous pseudo-metric $d$ on $X$ and let \begin{displaymath}
	D \defeq \sup\nolimits_{\alpha > 0} \liminf\nolimits_{i \to I} \sup\nolimits_{x \in X} \mathrm{ObsDiam}\! \left(G,d_{G,x},\mu_{i}; -\alpha\right) .
\end{displaymath} Since $G$ is amenable due to Theorem~\ref{theorem:topological.day}, there exists a $G$-invariant regular Borel probability measure $\nu$ on $X$. Note that, for every finite subset $E \subseteq G$ and every $\epsilon > 0$, there exists some $x \in X$ with $\sup_{g \in E} d(x,gx) < D + \epsilon$. Otherwise, there would exist a finite subset $E \subseteq G$ and some $\epsilon > 0$ such that $\sup_{g \in E} d(x,gx) \geq D + \epsilon$ for all $x \in X$, whence \begin{displaymath}
	\int \sup\nolimits_{g \in E} d(x,gx) \, d\nu(x) \geq D + \epsilon ,
\end{displaymath} contradicting the conclusion of Theorem~\ref{theorem:observable.diameters}. Appealing to the compactness of $X$, hence we deduce the existence of a point $x \in X$ with $\sup_{g \in G} d(x,gx) \leq D$. \end{proof}

We conclude this section with some remarks about further consequences of Theorem~\ref{theorem:observable.diameters}. For this purpose, we briefly clarify the connection between observable diameters and the L\'evy property in uniform spaces, cf.~\cite[Definition~2.6]{pestov02}.

\begin{definition} Let $X$ be a uniform space. For an entourage $U$ of $X$, let \begin{displaymath}
	U[A] \defeq \{ y \in X \mid \exists x \in A \colon \, (x,y) \in U \}\qquad (A \subseteq X) .
\end{displaymath} A net $(\mu_{i})_{i \in I}$ of Borel probability measures on $X$ is called a \emph{L\'evy net} in $X$ if, for every family $(B_{i})_{i \in I}$ of Borel subsets of $X$ and any open entourage $U$ of $X$, \begin{displaymath}
	\qquad \qquad \liminf\nolimits_{i \to I} \mu_{i}(B_{i}) > 0 \quad \Longrightarrow \quad \lim\nolimits_{i \to I} \mu_{i}(U[B_{i}]) = 1 .
\end{displaymath}  \end{definition}

Let us recall that every measurable real-valued function $f \colon X \to \mathbb{R}$ on a probability measure space $(X,\mathscr{B},\mu)$ admits a (not necessarily unique) \emph{median}, i.e., a real number $m \in \mathbb{R}$ with \begin{displaymath}
\mu (\{ x \in X \mid f(x) \geq m \}) \, \geq \, \tfrac{1}{2} \, \leq \, \mu (\{ x \in X \mid f(x) \leq m \}) .
\end{displaymath} We will need the following well-known fact. 

\begin{lem}[\cite{GromovMilman}, 2.5]\label{lemma:concentration} Let $X$ be a uniform space, let $d$ be a uniformly continuous pseudo-metric on $X$, and let $(\mu_{i})_{i \in I}$ be a L\'evy net of Borel probability measures on $X$. For each pair $(i,f) \in I \times \mathrm{Lip}_{1}(X,d)$, let $m_{i}(f)$ be a median of $f$ with respect to $\mu_{i}$. For every $\epsilon > 0$, \begin{displaymath}
	\sup\nolimits_{f \in \mathrm{Lip}_{1}(X,d)} \mu_{i}(\{ x \in X \mid \vert f(x) - m_{i}(f) \vert > \epsilon \}) \, \longrightarrow \, 0 \quad (i \to I) .
\end{displaymath} \end{lem}

\begin{proof} We include the proof for the sake of convenience. Let $\epsilon > 0$. Since $d$ is uniformly continuous, there is a symmetric open entourage $U$ of $X$ such that $d(x,y) \leq \epsilon$ for all $(x,y) \in U$. For all $(i,f) \in I \times \mathrm{Lip}_{1}(X,d)$, we conclude that $U[A_{i}(f)] \subseteq B_{i}(f)$ and $U[A'_{i}(f)] \subseteq B'_{i}(f)$ where \begin{align*}
	& A_{i}(f) \defeq \{ x \in X \mid f(x) \leq m_{i}(f) \} , & B_{i}(f) \defeq \{ x \in X \mid f(x) \leq m_{i}(f) + \epsilon \} , \\
	& A'_{i}(f) \defeq \{ x \in X \mid f(x) \geq m_{i}(f) \} , & B'_{i}(f) \defeq \{ x \in X \mid f(x) \geq m_{i}(f) - \epsilon \} .
\end{align*} Hence, $U[A_{i}(f)] \cap U[A'_{i}(f)] \subseteq B_{i}(f) \cap B'_{i}(f)$ for all $(i,f) \in i \times \mathrm{Lip}_{1}(X,d)$. We will show that \begin{displaymath}
	\sup\nolimits_{f \in \mathrm{Lip}_{1}(X,d)} \mu_{i}(X\setminus (B_{i}(f) \cap B'_{i}(f))) \, \longrightarrow \, 0 \quad (i \to I) .
\end{displaymath} Let $\delta > 0$. For each pair $(i,f) \in I \times \mathrm{Lip}_{1}(X,d)$, our hypothesis on $m_{i}(f)$ asserts that \begin{displaymath}
	\min \{ \mu_{i}(A_{i}(f)),\, \mu_{i}(A'_{i}(f)) \} \, \geq \, \tfrac{1}{2} .
\end{displaymath} Since $(\mu_{i})_{i \in I}$ is a L\'evy net in $X$, there exists $i_{0} \in I$ such that \begin{displaymath}
	\forall i \in I , \, i \geq i_{0} \ \forall f \in \mathrm{Lip}_{1}(X,d) \colon \qquad \mu_{i}(U[A_{i}(f)]) \, \geq \, 1 - \tfrac{\delta}{2} .
\end{displaymath} Otherwise, the subset \begin{displaymath}
	\left\{ i \in I \left| \, \exists f \in \mathrm{Lip}_{1}(X,d) \colon \, \mu_{i}(U[A_{i}(f)]) < 1 - \tfrac{\delta}{2} \right\} \right.
\end{displaymath} would be cofinal in $I$, which is easily seen to contradict the L\'evy property. Likewise, there exists some $i_{1} \in I$ with $i_{1} \geq i_{0}$ such that \begin{displaymath}
	\forall i \in I , \, i \geq i_{1} \ \forall f \in \mathrm{Lip}_{1}(X,d) \colon \qquad \mu_{i}(U[A'_{i}(f)]) \, \geq \, 1 - \tfrac{\delta}{2} .
\end{displaymath} Consequently, if $i \in I$ with $i \geq i_{1}$, then \begin{displaymath}
	\mu_{i} (B_{i}(f) \cap B'_{i}(f)) \, \geq \, \mu_{i}(U[A_{i}(f)] \cap U[A_{i}'(f)]) \, \geq \, 1 - \delta
\end{displaymath} for every $f \in \mathrm{Lip}_{1}(X,d)$, which means that $\sup\nolimits_{f \in \mathrm{Lip}_{1}(X,d)} \mu_{i}(X\setminus (B_{i}(f) \cap B'_{i}(f))) \leq \delta$.  \end{proof}

The following is a fairly well known fact about $mm$-spaces, see e.g.~\cite[Proposition~5.7]{ShioyaBook}, adapted to uniform spaces in a straightforward manner.

\begin{prop}\label{proposition:levy} Let $(\mu_{i})_{i \in I}$ be a net of Borel probability measures on a uniform space $X$. The following are equivalent. \begin{itemize}
	\item[$(1)$] $(\mu_{i})_{i \in I}$ is a L\'evy net in $X$. \smallskip
	\item[$(2)$] For every uniformly continuous pseudo-metric $d$ on $X$ and every $\alpha > 0$, \begin{displaymath}
					\qquad \lim\nolimits_{i \to I} \mathrm{ObsDiam}(X,d,\mu_{i}; -\alpha) \, = \, 0.
				\end{displaymath}
	\item[$(3)$] For every bounded uniformly continuous pseudo-metric $d$ on $X$ and every $\alpha > 0$, \begin{displaymath}
					\qquad \lim\nolimits_{i \to I} \mathrm{ObsDiam}(X,d,\mu_{i}; -\alpha) \, = \, 0.
				\end{displaymath}
\end{itemize} \end{prop}

\begin{proof} (1)$\Longrightarrow$(2). Consider a uniformly continuous pseudo-metric $d$ on $X$. For every $i \in I$ and $f \in \mathrm{Lip}_{1}(X,d)$, let $m_{i}(f)$ be a median of $f$ with respect to $\mu_{i}$. Let $\alpha > 0$. We show that \begin{displaymath}
	\lim\nolimits_{i \to I} \mathrm{ObsDiam}(X,d,\mu_{i}; -\alpha) = 0 .
\end{displaymath} Let $\epsilon > 0$. By~(1) and Lemma~\ref{lemma:concentration}, there exists $i_{0} \in I$ such that \begin{equation}\tag{$\ast$}\label{concentration}
	\forall i \in I, \, i \geq i_{0} \colon \quad \sup\nolimits_{f \in \mathrm{Lip}_{1}(X,d)} \mu_{i}\left( \left\{ x \in X \left\vert \, \vert f(x) - m_{i}(f) \vert \geq \tfrac{\epsilon}{2} \right\} \right) \, \leq \, \alpha . \right.
\end{equation} We argue that \begin{displaymath}
	\forall i \in I, \, i \geq i_{0} \colon \qquad \mathrm{ObsDiam}(X,d,\mu_{i};-\alpha) \, \leq \, \epsilon .
\end{displaymath} Let $i \in I$ where $i \geq i_{0}$. If $f \in \mathrm{Lip}_{1}(X,d)$, then $B_{i}(f) \defeq B_{d_{\mathbb{R}}}\!\left(m_{i}(f),\tfrac{\epsilon}{2}\right)$ is a Borel subset of $\mathbb{R}$ with $\diam (B_{i}(f),d_{\mathbb{R}}) \leq \epsilon$ and \begin{displaymath}
	f_{\ast}(\mu_{i})(B_{i}(f)) \, = \, \mu_{i} \!\left( \left\{ x \in X \left\vert \, \vert f(x) - m_{i}(f) \vert < \tfrac{\epsilon}{2} \right\} \right) \, \geq \, 1 - \alpha \right.
\end{displaymath} thanks to~\eqref{concentration}, wherefore $\mathrm{PartDiam}(f_{\ast}(\mu_{i}),1-\alpha) \leq \epsilon$. This completes the argument.

(2)$\Longrightarrow$(3). Trivial.

(3)$\Longrightarrow$(1). Let $U$ be an open entourage of $X$. According to classical work of Weil~\cite{weil}, there exist $\delta > 0$ and a uniformly continuous pseudo-metric $d$ on $X$ with $\diam (X,d) \leq 1$ such that $\{ (x,y) \in X \times X \mid d(x,y) < \delta \} \subseteq U$. Consider a family $(B_{i})_{i \in I}$ of Borel subsets of $X$ such that $\sigma \defeq \liminf_{i \to I} \mu_{i}(B_{i}) > 0$. In order to verify that $\lim_{i \to I} \mu_{i}(U[B_{i}]) = 1$, let $\epsilon > 0$. Fix any $i_{0} \in I$ with $\inf \{ \mu_{i}(B_{i}) \mid i \in I, \, i \geq i_{0} \} > \tfrac{\sigma}{2}$. By~(3), there is $i_{1} \in I$ with $i_{1} \geq i_{0}$ and \begin{equation}\tag{$\ast \ast$}\label{diameter}
	\forall i \in I, \, i \geq i_{1} \colon \qquad \mathrm{ObsDiam}\!\left(X,d,\mu_{i};-\min\!\left\{ \epsilon, \tfrac{\sigma}{2} \right\} \right) \, < \, \delta .
\end{equation} We argue that \begin{displaymath}
	\forall i \in I, \, i \geq i_{1} \colon \qquad \mu_{i}(U[B_{i}]) \, \geq \, 1 - \epsilon .
\end{displaymath} To this end, let $i \in I$ with $i \geq i_{1}$. Note that $f_{i} \colon X \to \mathbb{R}, \, x \mapsto \inf\{ d(x,y) \mid y \in B_{i} \}$ belongs to $\mathrm{Lip}_{1}(X,d)$. Hence, by~\eqref{diameter}, there exists a Borel subset $C_{i} \subseteq X$ with $\mu_{i}(C_{i}) \geq 1 -\min \{ \epsilon, \tfrac{\sigma}{2} \}$ and $\diam (f_{i}(C_{i}),d_{\mathbb{R}}) < \delta$. The former implies that $\mu_{i}(B_{i} \cap C_{i}) > 0$ and thus $B_{i} \cap C_{i} \ne \emptyset$, wherefore $0 \in f_{i}(C_{i})$. Since $\diam (f_{i}(C_{i}),d_{\mathbb{R}}) < \delta$, we now conclude that $f_{i}(C_{i}) \subseteq (-\delta,\delta)$ and hence $C_{i} \subseteq B_{d}(B_{i},\delta) \subseteq U[B_{i}]$. It follows that $\mu_{i}(U[B_{i}]) \geq \mu_{i}(C_{i}) \geq 1-\epsilon$, as desired. \end{proof}

As any bounded right-uniformly continuous pseudo-metric on a topological group is bounded from above by a bounded continuous right-invariant pseudo-metric, we arrive at the following characterization of the L\'evy property on topological groups (with their right uniformity).

\begin{cor}\label{corollary:levy} Let $(\mu_{i})_{i \in I}$ be a net of Borel probability measures on a topological group $G$. The following are equivalent. \begin{itemize}
	\item[$(1)$] $(\mu_{i})_{i \in I}$ is a L\'evy net in $G$. \smallskip
	\item[$(2)$] For every continuous right-invariant pseudo-metric $d$ on $G$, \begin{displaymath}
					\qquad \sup\nolimits_{\alpha > 0} \lim\nolimits_{i \to I} \mathrm{ObsDiam}(G,d,\mu_{i}; -\alpha) = 0.
				\end{displaymath}
	\item[$(3)$] For every bounded continuous right-invariant pseudo-metric $d$ on $G$, \begin{displaymath}
					\qquad \sup\nolimits_{\alpha > 0} \lim\nolimits_{i \to I} \mathrm{ObsDiam}(G,d,\mu_{i}; -\alpha) = 0.
				\end{displaymath}
\end{itemize} \end{cor}

In view of Corollary~\ref{corollary:levy}, let us point out two consequences of our results. For one thing, Corollary~\ref{corollary:observable.diameters} yields a quantitative version of~\cite[Theorem~7.1]{GromovMilman}, i.e., the extreme amenability of L\'evy groups. And for another thing, Theorem~\ref{theorem:observable.diameters} readily implies~\cite[Theorem~3.9]{PestovSchneider}, an extension of the result for Polish groups by Pestov~\cite[Theorem~5.7]{pestov10} generalizing earlier work of Glasner, Tsirelson and Weiss~\cite[Theorem~1.1]{GlasnerTsirelsonWeiss}.

\begin{cor}[\cite{PestovSchneider}, Theorem~3.9]\label{corollary:whirly} If a topological group $G$ admits a L\'evy net of Borel probability measures UEB-converging to invariance over $G$, then $G$ is whirly amenable, i.e., \begin{enumerate}
	\item[$\bullet$] $G$ is amenable, and
	\item[$\bullet$] every invariant regular Borel probability measure on a $G$-flow is supported on the set of fixed points.
\end{enumerate} \end{cor}

\begin{proof} The amenability of $G$ is due to Theorem~\ref{theorem:topological.day}, while the second assertion is an immediate consequence of Theorem~\ref{theorem:observable.diameters} combined with Corollary~\ref{corollary:levy} and Remark~\ref{remark:monotonicity}. \end{proof}

Let us finish with an open problem.

\begin{remark} Since the L\'evy property can be stated in the more general framework of uniform spaces, it would be very interesting to know if Gromov's concentration topology admits an equally natural extension in that context. If so, then one may hope to generalize Theorem~\ref{theorem:equivariant.concentration} to the case of topological groups with non-metrizable universal minimal flow. \end{remark}

\section*{Acknowledgments}

This research has been supported by funding of the Excellence Initiative by the German Federal and State Governments as well as the Brazilian Conselho Nacional de Desenvolvimento Cient\'{i}fico e Tecnol\'{o}gico (CNPq), processo 150929/2017-0. The author is deeply indebted to Vladimir Pestov for a number of inspiring and insightful discussions about Gromov's concentration topology as well as his most valuable comments on an earlier version of this paper. Furthermore, the kind hospitality of CFM--UFSC (Florian\'opolis) and IME--USP (S\~{a}o Paulo) during the origination of this work is gratefully acknowledged.


\end{document}